\documentclass{amsart}
\usepackage{cite}
\usepackage{tikz-cd}
\usepackage[all]{xy}
\usepackage{graphicx}
\usepackage{amsmath,amssymb}
\usepackage{amssymb}
\usepackage{mathrsfs}
\usepackage{graphicx}
\usepackage{subfigure}
\usepackage{extarrows}
\usepackage{verbatim}
\usepackage{placeins}[section]
\usepackage[numbers,sort&compress]{natbib} % ʹ�ο������������ŵ�������ѹ����ʽ����
\usepackage[bookmarksnumbered, bookmarksopen,
colorlinks,citecolor=blue,linkcolor=blue]{hyperref}
\newcommand{\seRightarrow}{\mathrel{\rotatebox[origin=c]{-45}{$\Rightarrow$}}}

\newtheorem{Theorem}{Theorem}[section]
\newtheorem{Lemma}[Theorem]{Lemma}
\theoremstyle{Definition}
\newtheorem{Definition}[Theorem]{Definition}
\newtheorem{Example}[Theorem]{Example}
\newtheorem{Proposition}[Theorem]{Proposition}
\newtheorem{Corollary}[Theorem]{Corollary}

\theoremstyle{Remark}
\newtheorem{Remark}[Theorem]{Remark}

\theoremstyle{notation}

\numberwithin{equation}{section}

 \def\3{\operatorname{3}}
\begin{document}

\title{Discrete homotopy and homology theories for finite posets}

\author{Jing-Wen Gao   }
\address{School of Mathematics and Statistics, Huazhong University of Science and Technology, Wuhan 430074, P.R.China}

\email{d202280005@hust.edu.cn (Jing-Wen Gao),\quad yangxs@hust.edu.cn (Xiao-Song Yang)}
%    \thanks will become a 1st page footnote.
%\thanks{This research did not receive any specific grant from funding agencies in the public, commercial, or not-for-profit sectors.}

\author{Xiao-Song Yang  }
\address{}
\email{ }

\thanks{ } %and the Scientific Research Foundation for the Returned Overseas Chinese Scholars, State Education Ministry.}

%    General info

\date{\today}

%\dedicatory{This paper is dedicated to our advisors.}


\begin{abstract}
This paper presents a discrete homotopy theory and a discrete homology theory for finite posets. In particular, the discrete and classical homotopy groups of finite posets are always isomorphic.
 Moreover, this discrete homology theory is related to the discrete homotopy theory through a discrete analogue of the Hurewicz map.
\end{abstract}
\maketitle
\section{Introduction}
The classical homotopy and homology theories for finite posets are based on a well-known correspondence given by McCord, which assigns to each finite poset $X$ the order complex $\mathcal{K}(X)$ together with a weak homotopy equivalence $|\mathcal{K}(X)|\rightarrow X$ {\rm{\cite[Theorem 1.4.6]{JAB2011}}}. This correspondence permits the study of homotopy and homology groups of a finite poset by means of its associated order complex. 

The classical homotopy-theoretic invariants for a finite poset involve only its topological structure, that is, the associated order complex. However,  
a salient  feature of a finite poset is that it is essentially a combinatorial object, its graphical representation (Hasse diagram) being a digraph. The combinatorial properties of posets have been applied to other areas, such as topological data analysis. For example, in {\rm{\cite{JFJ2019}}} Jardine utilizes poset theoretical ideas to study Vietoris-Rips complexes associated with metric spaces. Moreover, in light of the observation
  that every finite digraph can be viewed as the graphical representation (Hasse diagram) of a finite poset,
 the study of the combinatorics of posets can facilitate the understanding of the combinatorics of digraphs. The homotopy-theoretic invariants {\rm{\cite{HB2001, HB2005}}}, based on combinatorics of the latter, have found numerous applications, including in subspace arrangements {\rm{\cite{HB2011}}}, pattern recognition {\rm{\cite{VC2013}}} and topological data analysis {\rm{\cite{FM2025}}}. Hence,
 the motivation for the introduction of a new discrete homotopy theory for finite posets in this paper comes from a desire to find invariants that are sensitive to the combinatorics encoded in the posets.

The goal of this paper is to develop a discrete  homotopy theory and a discrete homology theory, which are discrete analogues of homotopy and homology theories, associating a 
sequence of groups to a finite poset, capturing its combinatorial structure, rather than its topological structure. Furthermore,  discrete homology theory  is related to discrete homotopy theory in the same way that classical homology is related to classical homotopy theory. Since every finite digraph can be viewed as the Hasse diagram of a finite poset, our theories provide much algebraic insight into problems regarding the combinatorics of finite digraphs.

We prove three main results. The first is motivated by a discrete homotopy theory for graphs introduced in {\rm{\cite{EB2006, HB2001, BL2019}}}. In similar fashion we define a discrete homotopy theory for finite posets. Our first result  asserts that, for all finite posets, the discrete homotopy theory and the classical one lead to isomorphic homotopy groups in all dimensions (Theorem \ref{homotopy groups}). This allows us to study the classical homotopy groups of compact polyhedra from a new viewpoint, the weak homotopy types of compact polyhedra being in one to one correspondence with the weak homotopy types of finite posets {\rm{\cite[Theorem 1.4.6, Theorem 1.4.12]{JAB2011}}}. 

In general, the classical homotopy groups of topological spaces are surprisingly difficult to compute. For example, covering spaces are used to calculate the fundamental group of a circle, one of the most foundational calculations in algebraic topology. To illustrate how the methods of calculating discrete homotopy groups of some topological spaces can  be simpler and more tractable than the ones used to get the classical homotopy groups, we give a new method for calculating the fundamental group of a circle  by means of this new homotopy theory (Example \ref{circle}).  

The motivation for the second result comes from a discrete cubical homology theory for graphs introduced in {\rm{\cite{HB2019, HB2014}}}. We define a similar discrete cubical homology theory for finite posets. Our second result  is that we exhibit for each $n\geq 1$ a natural homomorphism $\psi: H_n^{\rm Cube}(X)\rightarrow  H_n^{\rm Simpl}(\mathcal{K}(X))$, where $H_n^{\rm Cube}(X)$ is the $n$-th discrete cubical homology group, and a class of finite posets $X$ for which $\psi$ is surjective {\rm (Theorem \ref{Main1})}. We provide an example for which the homology groups of the two theories are not isomorphic at least in dimension $1$ (Example \ref{con}).
Theorem \ref{Main1} has an interesting application as follows. Let $M=|K|$ be a triangulated $m$-manifold with $K$ simplicially collapsible. 
Proposition \ref{apptom} tells when the manifold $M$ with two specific balls removed is not contractible.

The third result makes progress towards  the discrete analogue of the Hurewicz map relating discrete homology theory to discrete homotopy theory, motivated by the one defined  in {\rm{\cite{ BL2019}}}. Similarly, we give a notion of discrete Hurewicz map.
There is a discrete Hurewicz theorem in dimension 1 (Theorem \ref{di1}). Continuing our analogy with classical homotopy and homology theories, 
our third result  asserts that this discrete Hurewicz map coincides with the classical one (Theorem \ref{Hurewicz coincide}). 

The organization of the paper is as follows. In Section \ref{section2}, we give preliminary definitions and results concerning finite posets and simplicial complexes.
In Section \ref{section3}, we define a discrete homotopy theory for finite posets and show that the discrete and classical homotopy groups of finite posets agree in all dimensions.
In Section \ref{section4}, we define a discrete cubical homology theory for finite posets, and explore both the similarities and differences between the discrete cubical homology theory and the classical one.  In Section \ref{section5}, we define a discrete Hurewicz map  which is the same as the classical one.

\section{Preliminaries}\label{section2}
This section is devoted to introducing some definitions and properties of finite posets  and simplicial complexes that are needed in the sequel. We refer the reader to {\rm{\cite{ JAB2011, JFP1969, EGM2012,  CPR1972, ECZ1963}}} for further details or proofs.
\subsection{Finite posets}
A topological space is said to be Alexandroff if every intersection of open sets is
still an open set. If $X$ is an Alexandroff space, for each $x\in X,$ we denote by $U_x$
the intersection of all the open sets of $X$ containing $x$. This allows us to define a
reflexive and transitive relation on $X$ by declaring $x \leq y$ if $U_x\subseteq U_y.$ The relation is antisymmetric if and only if $X$ is $T_0$. On the other hand, if $X$ is a set endowed with a reflexive, transitive relation, then the sets $U_x = \{y\in X | y\leq x\}$, as a basis, determines an Alexandroff topology. The Alexandroff topology is $T_0$ if and only if the relation is antisymmetric, that is, a partial order. A set with a partial order is called a poset. It has been shown that Alexandroff $T_0$ spaces and posets are in bijective correspondence. A finite $T_0$ space is clearly an Alexandroff $T_0$ space, thus a finite $T_0$ space is equivalent to a finite poset.

A  useful way to represent posets is with their \emph{Hasse diagrams}. The Hasse diagram of a poset $X$ is a digraph whose vertices are the points of $X$ and whose edges are the ordered pairs $(x,y)$ such that $x<y$ and there exists no $z\in X$ such that $x<z<y$.

Given a  poset $X$ and an element $x\in X$, we denote $
U_x=\{y\in X~|~y\leq x\},~ \widehat{U}_x=\{y\in X~|~y<x\},$ 
$F_x=\{y\in X~|~x\leq y\},$ and $\widehat{F}_x=\{y\in X~|~x<y\}.$ Let $C_x=U_x\cup F_x$ and $\widehat C_x=\widehat U_x\cup \widehat F_x$.
A  totally ordered subset $C$ of $X$ is called a \emph{chain}. An $n$-chain of $X$ is a chain of cardinality  $n+1$.  

\begin{Definition}
\rm{A finite poset $X$ is \emph{homogeneous} of dimension $n$ if all maximal chains in $X$
 have cardinality  $n + 1$. A poset $X$ is \emph{graded} if $U_x$ is homogeneous for all $x\in X$. In this case, the \emph{degree} of $x$, denoted ${\rm deg}(x)$, is the dimension of $U_x$.
}
\end{Definition}

If $X$ is a poset, we define its \emph{order complex} $\mathcal{K}(X)$ as the ordered simplicial complex whose simplices are the nonempty chains of $X$. Conversely,
given a simplicial complex $K$,  we define the \emph{face poset} $\chi(K)$ of $K$ as the poset of simplices of $K$ ordered by inclusion.
\begin{Definition}
\rm{Let $X$ be a finite poset. Define the \emph{K-McCord  map} 
	$\mu_X : |\mathcal{K}(X)|\rightarrow X$ by $\mu_X(u)=\text{min}(\text{support}(u))$.
}
\end{Definition}

Let $f, g: X\rightarrow Y$ be two maps between finite posets. The following statements hold.
\begin{Proposition}\label{map}
\begin{itemize}
	\item [(1)] $f$  is continuous if and only if it is order-preserving.
	\item [(2)] $f$ is homotopic to $g$ if and only if there is a sequence $f_0, f_1,\cdots, f_n$ of order-preserving maps such that $f=f_0\leq f_1\geq f_2\leq\cdots f_n=g$.
\end{itemize}
\end{Proposition}
Henceforth, all maps between posets that we deal with are assumed to be order-preserving.
\subsection{Simplicial complexes}
\begin{Definition}
\rm{ Let $K$ be a simplicial complex. A simplex $A\subseteq K $ is called a \emph{maximal simplex} if $A$ is not a proper face of any simplex of $K$. The face $B$ of $A$ is called a \emph{free face} of $A$ in $K$ if $B$ is the proper face of no other simplex of $K$.

$K$ is a \emph{homogeneous $k$-complex} if every maximal simplex of $K$ is $k$-dimensional.}
\end{Definition}

\begin{Definition}
\rm{Let $L\subseteq K$ be simplicial complexes. We say there is an \emph{elementary collapse} from $K$ to $L$ if $K-L$ consists of a maximal simplex $A$ of $K$ together with a free face, and write $K\seRightarrow L$.

We say $K$ \emph{collapses} to $L$, written $K\searrow L$, if there is a finite sequence of elementary collapses $K=K_0\seRightarrow K_1\seRightarrow\cdots\seRightarrow K_n=L$. If $L$ is a point, we call $K$ \emph{ collapsible} and write $K\searrow 0$.}
\end{Definition}

\begin{Definition}
\rm{	Suppose $Q\subseteq P$ are polyhedra and that $P=Q\cup B^n$ and $Q\cap B^n=$ a face $B^{n-1}$, where $B^i$ denotes a piecewise-linear (PL) $i$-ball. Then we say that there is an \emph{elementary collapse} of $P$ onto $Q$, and write $P\seRightarrow Q$. The collapse is across $B^n$ onto $B^{n-1}$ from the complementary face ${\rm cl}(\partial B^n-B^{n-1})$.
	
	We say $P$ \emph{collapses} to $Q$, written $P\searrow Q$, if there is a finite sequence of elementary collapses $P=P_0\seRightarrow P_1\seRightarrow\cdots\seRightarrow P_n=Q$. If $Q$ is a point, we call $P$ \emph{collapsible} and write $P\searrow0$.}
\end{Definition}

It is true that $K\searrow L$ implies $|K|\searrow|L|$, but the converse is unknown.

\section{Discrete homotopy theory for finite posets}\label{section3}
In this section we define  a discrete homotopy theory for finite posets, motivated by the notion of a discrete homotopy theory for graphs introduced in {\rm{\cite{HB2001, BL2019}}}.
As we shall see in Theorem \ref{homotopy groups}, the discrete and classical homotopy groups of finite posets agree in all dimensions.

Let $\mathbb{Z}$ denote the digraph whose vertices are  integers and whose edges are the ordered pairs $(2i, 2i+1)$ and $(2j, 2j-1)$.
Given an integer $p\geq 0$, write $I_p$ for the subgraph of $\mathbb{Z}$ induced by $\{0,\cdots, p\}$. 

Two  maps $f, g: X\rightarrow Y$ between finite posets are \emph{homotopic} if for some $p$, there is a map $H: X\times I_p\rightarrow Y$ such that $H(x, 0)=f(x)$ and $H(x, p)=g(x)$. The map $H$ is a \emph{homotopy} from $f$ to $g$. Write $H_i(x)=H(x, i)$.
This definition in fact coincides with Proposition \ref{map}.

Let $\mathbb{Z}^n$ be the $n$-fold Cartesian product of $\mathbb{Z}$ whose vertices are $n$-tuples of integers.
For any interval $J\subseteq\mathbb{R}$,  $\mathbb{Z}^n_J$ denotes the subgraph  consisting of all vertices $x=(x_1,\cdots, x_n)\in\mathbb{Z}^n$ with $x_i\in J$. Let $I^n_{p}$ be the $n$-fold Cartesian product of $I_{p}$.
Denote by  $\partial I^n_{p}$ the subgraph of $I^n_{p}$ consisting of all vertices $x=(x_1,\cdots, x_n)\in I^n_{p}$ with $x_i=p$ or $0$ for some $i$.

By abuse of notation, we will write 
$$f: (\mathbb{Z}^n, \partial\mathbb{Z}^n)\rightarrow (X, x_0)$$
if there exists an integer $r\geq 0$  such that $f$ is a map $f: (\mathbb{Z}^n, \mathbb{Z}_{[r, +\infty]}^n)\rightarrow (X, x_0)$. The minimum of such $r$ is called the \emph{radius} of $f$. 
\begin{figure}[htbp]
	\centering
	\includegraphics[scale=0.3]{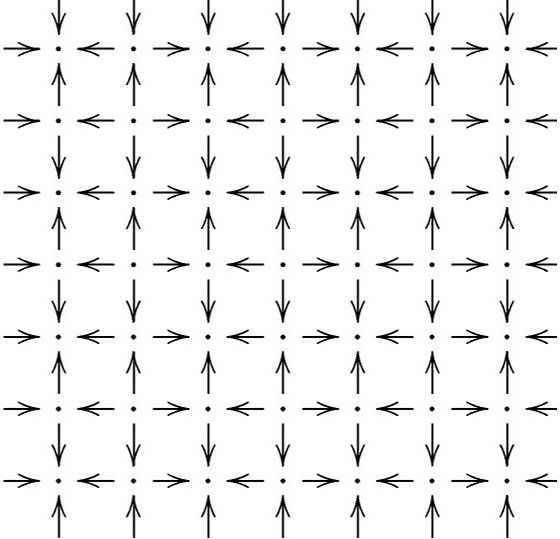}
	\caption{The Hasse diagram of $\mathbb{Z}^2$.}
	\label{figure}
\end{figure}

For a finite poset $X$ with basepoint $x_0\in X$, two   maps $f, g: (\mathbb{Z}^n, \partial\mathbb{Z}^n)\rightarrow (X, x_0)$ are  homotopic if there is a \emph{basepoint-preserving homotopy} $H: (\mathbb{Z}^n, \partial\mathbb{Z}^n)\times I_p\rightarrow X$ from $f$ to $g$ such that $H_i$ is a map  $(\mathbb{Z}^n, \partial\mathbb{Z}^n)\rightarrow (X, x_0)$ for all $i$. The basepoint-preserving homotopy class of $f$ is denoted $[f]$.

\begin{Definition}{\rm
		For a finite poset $X$ with basepoint $x_0\in X$, define $\pi_n^{D}(X, x_0)$ to be the set of basepoint-preserving-homotopy classes of  maps $f: (\mathbb{Z}^n, \partial\mathbb{Z}^n)\rightarrow (X, x_0)$. We call $\pi_n^{D}(X, x_0)$ the \emph{$n$-th discrete homotopy group} of $X$.
	} 
\end{Definition}

We can endow $\pi^{D}_n(X, x_0)$ with a group structure as follows. If $f$ and $g$ are maps $(\mathbb{Z}^n, \partial\mathbb{Z}^n)\rightarrow (X, x_0)$ of radii $r_f$ and $r_g$ respectively, then the product $f\cdot g$ is defined by
\begin{equation*}
(f\cdot g)(x_1,\cdots, x_n)=\begin{cases}
f(x_1,\cdots, x_n)~&\text{if}~x_1\leq r_f,\\
g(x_1-(r_f+r_g), x_2,\cdots, x_n)~&\text{if}~x_1>r_f.
\end{cases}
\end{equation*}
The identity element is the constant map sending $\mathbb{Z}^n$ to $x_0$.

Similar arguments as for discrete homotopy groups of graphs in {\rm{\cite{HB2001}}} show that the definition of $\pi^{D}_n(X, x_0)$ is independent of the basepoint $x_0$ when $X$ is connected, and that  $\pi^{D}_n(X, x_0)$ is a group under the product defined above. Furthermore, the idea of the proof of {\rm{\cite[Proposition 3.6]{HB2001}}} can be used to show that $\pi^{D}_n(X, x_0)$ is abelian for $n\geq 2$.

Before we construct a homomorphism from the discrete $n$-th homotopy group $\pi^{D}_n(X, x_0)$ to the classical $n$-th homotopy group $\pi_n(|\mathcal{K}(X)|, x_0)$, we make a few preliminary observations. 

Let $L^n$ be the $n$-dimensional  cube, the product of $n$ copies of the interval $[-1, 1]$.
For each $i$-th copy of $L^n$, consider a partition
\begin{equation}\label{e1}
-1=t_{-{j_i}^{\prime}}<\cdots<t_{-1}<0=t_0<t_{1}<\cdots<t_{j_i}=1,
\end{equation}
where ${j_i}^{\prime}$ and $j_i$ are the numbers of partition points between $[-1, 0)$ and $(0, 1]$, respectively.
Then we define a partition $T$ of $L_n$ by
\begin{equation}\label{pp1}
T=\bigcup_{\substack{t_{-{j_i}^{\prime}}\leq k\leq t_{j_i}\\ 1\leq i\leq n}}t_k.
\end{equation} With this partition $T$,
$L^n$ can be identified with the geometric realization of a subgraph $\gamma_T$ of $\mathbb{Z}^n$,
\begin{equation*}
\gamma_T=\prod_{i=1}^n\mathbb{Z}_{[-{j_i}^{\prime}, j_i]}.
\end{equation*}
For each vertex $(x_1, x_2,\cdots, x_n)$ of $\gamma_T$, taking it to the point $(t_{x_1}, t_{x_2},\cdots, t_{x_n})$ of $L^n$
induces a canonical homeomorphism
\begin{align*}
\lambda_T: |\mathcal{K}(\gamma_T)|&\longrightarrow L^n.
\end{align*}
When the partition $T$ takes the form
$$-1=\frac{-r}{r}<\frac{-r+1}{r}<\cdots<\frac{-1}{r}<0=t_0<\frac{1}{r}\cdots<\frac{r-1}{r}<\frac{r}{r}=1$$
for each $i$, where $r>0$ is an integer, we will write $T$, $\gamma_{T}$ and $\lambda_{T}$ as $T_r$, $\gamma_{r}$ and $\lambda_r$, respectively.

\begin{Definition}{\rm Let $X$ be a connected finite poset.
		Define $\theta: \pi_n^D(X, x_0)\rightarrow \pi_n(|\mathcal{K}(X)|, x_0)$ by the equation
		$$\theta([f])=[|\mathcal{K}(f)|\circ\lambda_{r_f}^{-1}],$$
		where $r_f$ is the radius of the map $f$.}
\end{Definition}
It is straightforward to check that $\theta$ is  well-defined and is a
 homomorphism.
The following is one of the main theorems of this paper.
\begin{Theorem}\label{homotopy groups}
	The homomorphism $\theta$ is an isomorphism.
\end{Theorem}
There is a certain amount of technical machinery to be set up in order to prove this result. Given a partition $T$ of $L^n$ defined as (\ref{pp1}), write $L^n$ as $\lambda_T|\mathcal{K}(\gamma_T)|$.
To each continuous map $f: \lambda_T|\mathcal{K}(\gamma_T)|\rightarrow |\mathcal{K}(X)|$ which takes every simplex of $\lambda_T|\mathcal{K}(\gamma_T)|$ into some simplex of $|\mathcal{K}(X)|$, we associate a quotient space $$J^f_T(|\mathcal{K}(X)|)$$ of $|\mathcal{K}(X)|$ by the following procedure. Consider two vertices $$b=\lambda_T(a),~b^{\prime}=\lambda_T(a^{\prime}) $$ in $\lambda_T|\mathcal{K}(\gamma_T)|$ with $a<a^{\prime}$. Then there is a simplex in $|\mathcal{K}(X)|$ containing $f(b)$ and $f(b^{\prime})$, denoted $\sigma$. Let $l_{f(b)f(b^{\prime})}$ be a simple curve in $\sigma$ connecting $f(b)$ and $f(b^{\prime})$, which is disjoint from the images of all the other vertices of $\lambda_T|\mathcal{K}(\gamma_T)|$ under $f$. Furthermore, $l_{f(b)f(b^{\prime})}$ can be chosen to satisfy that for every point $x$ in $l_{f(b)f(b^{\prime})}-\{f(b), f(b^{\prime})\}$,
\begin{align}\label{lc}
\text{support}(x)=\begin{cases}
\text{support}(f(b)), ~\text{if}~\mu_X(f(b))\leq \mu_X(f(b^{\prime}))\\
\text{support}(f(b^{\prime})),~\text{if}~\mu_X(f(b^{\prime}))\leq \mu_X(f(b)).
\end{cases}
\end{align}

 Let us form the quotient space $J^f_T(|\mathcal{K}(X)|)$ under the identifications 
\begin{align}\label{lcc}
x\sim\begin{cases}
f(b), ~\text{if}~\text{support}(x)= \text{support}(f(b)),\\
f(b^{\prime}),~\text{if}~\text{support}(x)= \text{support}(f(b^{\prime})),
\end{cases}
\end{align}
 where $x\in l_{f(b)f(b^{\prime})}-\{f(b), f(b^{\prime})\}$, and $b, b^{\prime}$ range over all the pairs which are the images under $\lambda_{T}$ of the comparable pairs in $\gamma_{T}$.
It is easy to see that by construction $\mu_X$ is well-defined on the quotient space $J^f_T(|\mathcal{K}(X)|)$. We write $\widetilde \mu_X$ for this induced map.

Before proceeding further we need some technical lemmas about the above constructions. 
\begin{Lemma}\label{continuu}
$\widetilde \mu_X: J^f_T(|\mathcal{K}(X)|)\rightarrow X$ is continuous.
\end{Lemma}
\begin{proof}
Let $f(b^{\prime})\in X$ and let $Z=\mathcal{K}(X-U^X_{f(b^{\prime})})$, then by {\rm{\cite[Theorem 1.4.6]{JAB2011}}} $$\mu_X^{-1}(U^X_{f(b^{\prime})})=|\mathcal{K}(X)|-|Z|.$$
Note that (\ref{lc}) and (\ref{lcc}) imply
 $$(l_{f(b)f(b^{\prime})}-f(b))\cap |Z|=\emptyset.$$
Thus $\pi^{-1}\circ\widetilde\mu_X^{-1}(U^X_{f(b^{\prime})})=\mu_X^{-1}(U^X_{f(b^{\prime})})$ is open, where $\pi: |\mathcal{K}(X)|\rightarrow J^f_T(|\mathcal{K}(X)|)$ is the projection.
\end{proof}

\begin{Lemma}\label{continu}
The map $$J_{T}(f):\gamma_{T}\rightarrow J^f_T(|\mathcal{K}(X)|)$$ defined by $J_{T}(f)(a)=[f(\lambda_T(a))]$ is continuous.
\end{Lemma}
\begin{proof}
We may assume that $n=1$ here, since the topology on $\gamma_{T}$ is given by product topology. Given a point $[f(\lambda_T(a))]\in J^f_T(|\mathcal{K}(X)|)$, let $B_{[f(\lambda_T(a))]}\subseteq J^f_T(|\mathcal{K}(X)|)$ be a ball centered at $[f(\lambda_T(a))]$ with radius sufficiently small such that $$\pi^{-1}(B_{[f(\lambda_T(a))]})\cap f(\gamma_T)=\{f(\lambda_T(a))\}.$$
We will prove that $J_{T}(f)^{-1}(B_{[f(\lambda_T(a))]})$ is open which will shows that $J_{T}(f)$ is continuous.

Let $a$ is an odd integer. We can assume without loss of generality that $U^{\gamma_T}_a=\{a-1, a, a+1\}$. 
 If $\mu_X(f(\lambda_T(a-1)))\leq\mu_X(f(\lambda_T(a)))$, then $l_{f(\lambda_T(a-1))f(\lambda_T(a))}-\{f(\lambda_T(a))\}$ is identified with $f(\lambda_T(a-1))$. The fact $$\pi^{-1}(B_{[f(\lambda_T(a))]})\cap\{l_{f(\lambda_T(a-1))f(\lambda_T(a))}-\{f(\lambda_T(a))\}\} \not=\emptyset$$ says that
 $$a-1\in J_{T}(f)^{-1}(B_{[f(\lambda_T(a))]}).$$ 
 
 On the other hand, if $\mu_X(f(\lambda_T(a)))\leq\mu_X(f(\lambda_T(a-1)))$, then $l_{f(\lambda_T(a-1))f(\lambda_T(a))}-\{f(\lambda_T(a-1))\}$ is identified with $f(\lambda_T(a))$. For any point $b\in l_{f(\lambda_T(a-1))f(\lambda_T(a))}-\{f(\lambda_T(a-1))\}$, we have
 $$B_{[f(\lambda_T(a))]}=B_{[b]}.$$ 
 Note that $\pi^{-1}(B_{[b]})$ is an open ball containing $b$ in $|\mathcal{K}(X)|$. When $b$ becomes arbitrarily close to $f(\lambda_T(a-1))$, $f(\lambda_T(a-1))\in\pi^{-1}(B_{[b]}) $. Then it follows that 
 $$a-1\in J_{T}(f)^{-1}(B_{[b]})=J_{T}(f)^{-1}(B_{[f(\lambda_T(a))]}).$$
 The same argument proves that
 $$a+1\in J_{T}(f)^{-1}(B_{[f(\lambda_T(a))]}).$$ Thus by the choice of $B_{[f(\lambda_T(a))]}$ we conclude that
 $$U^{\gamma_T}_a=J_{T}(f)^{-1}(B_{[f(\lambda_T(a))]})$$ which is the open set in $\gamma_T$.
 
 The case $a$ is an even integer can be handled by similar arguments. Hence $J_{T}(f)$ is continuous, as desired.
 \end{proof}
\begin{Remark}\label{az}
{\rm Composing the maps in the previous two lemmas yields an order-preserving map $$\widetilde \mu_X\circ J_{T}(f): \gamma_{T}\rightarrow X.$$
An especially important fact is that when the domain of $f$ is pulled back to $\gamma_{T}$ by the homeomorphism $\lambda_{T}$, $\mu_X\circ f$ and $\widetilde \mu_X\circ J_{T}(f)$ take the same values on every element of $\gamma_{T}$, that is,
$$ \mu_X\circ f\circ\lambda_{T}|\gamma_{T}=\widetilde \mu_X\circ J_{T}(f).$$
This allows $\widetilde \mu_X\circ J_{T}(f)$ to serve as a useful tool for the reconstruction and approximation of topological properties of $f$.

Hence, when $f\in \pi_n(|\mathcal{K}(X)|, x_0)$, one might hope that
$\widetilde \mu_X\circ J_{T}(f)$ could provide all necessary details to form a preimage of $f$ in $\pi_n^D(X, x_0)$.
This is indeed the case as we shall see.

}
\end{Remark}

 To get a better understanding of the proof of this theorem, let us sketch the proof strategy.
The main idea of proving this theorem is to pull back information about $X$ to $|\mathcal{K}(X)|$ via the powerful tool, the {\emph{K-McCord  map}} $\mu_X$. The injectivity part will not be difficult once we finish the proof of the surjectivity part. The harder part of the proof is to show $\theta$ is surjective. In order to achieve this, when given a map $f: L^n\rightarrow |\mathcal{K}(X)|$, we first reconstruct 
it in terms of a simpler one, a simplicial map $f^{\prime}$. As stated in Remark \ref{az}, we define a preimage $g: (\mathbb{Z}^n, \partial\mathbb{Z}^n)\rightarrow (X, x_0)$ in terms of $f^{\prime}$, namely, $$\theta([g])=[|\mathcal{K}(g)|\circ\lambda_{r_g}^{-1}]=[f].$$ 
The fact that $\mu_X$ is a weak homotopy equivalence allows us to turn to prove
$$[\mu_X\circ\mathcal{K}(g)\circ\lambda^{-1}_{r_g}]=[{\mu_X}\circ |f|].$$ To prove this,
a finite sequence of maps will be interpolated between them, each map in the sequence being comparable with its neighbors. We now start the proof of the theorem, filling in the technical details.
\begin{proof}[Proof of Theorem \ref{homotopy groups}]
	We first show that $\theta$ is surjective.
	Let $f\in \pi_n(|\mathcal{K}(X)|, x_0)$.  Using the simplicial approximation theorem {\rm{\cite[Theorem 16.1]{JRM1984}}},
	we can  represent $f$  by a simplicial map  $f^{\prime}$ from $K$ to $\mathcal{K}(X)$, where $K$ is some iterated barycentric subdivision of $L^n$. The key observation now is that for an arbitrary integer $r>0$, under the homeomorphism $\lambda_{r}$, $\mathcal{K}(\gamma_r)$ becomes a triangulation of $L^n$.
	When $r$ is sufficiently large, $\mathcal{K}(\gamma_r)$ can be regarded as a refinement of $K$.
	 
	Given a simplex $\sigma$ of $\mathcal{K}(X)$, let $|\sigma|_{\epsilon}$ be the open set in $|\mathcal{K}(X)|$  with the following properties:
	\begin{itemize}
		\item [(1)] $|\sigma|\subseteq|\sigma|_{\epsilon}$.
		\item [(2)] $|\sigma|_{\epsilon}$ contains no other vertices of $\mathcal{K}(X)$ besides those of $\sigma$. 
	\end{itemize}
	Cover $L^n$ by the open sets
	$$|f^{\prime}|^{-1}(|\sigma|_{\epsilon})$$
 as $\sigma$ ranges over all maximal simplices of $\mathcal{K}(X)$. Denote this open cover by $\mathcal{A}$.
	Let $\delta_{\mathcal{A}}$  be the Lebesgue number for this open cover $\mathcal{A}$.
	We take $r_1$ to be an even multiple of $r$ so that each subpolyhedron of 
	$$\lambda_{r_1}(|\mathcal{K}(\gamma_{r_1})|)$$
	has diameter less than $\frac{\delta}{2}$ where $\delta<\delta_{\mathcal{A}}$. By the choice of $r_1$, $\mathcal{K}(\gamma_{r_1})$ can be regarded as a refinement of $\mathcal{K}(\gamma_r)$.

	Consider a map $g: (\mathbb{Z}^n, \partial\mathbb{Z}^n)\rightarrow (X, x_0)$ given by
	\begin{equation*}
	g(w)=\begin{cases}
	{\widetilde \mu_X}\circ J_{T_{r_1}}(|f^{\prime}|)(w),~\text{if}~w\in \gamma_{{r_1}},\\
	x_0,~\text{otherwise.}
	\end{cases}
	\end{equation*} 
	$g$ is order-preserving by Lemma \ref{continuu} and Lemma \ref{continu}.
	We claim that
	\begin{equation}\label{surj}
	[\mu_X\circ |\mathcal{K}(g)|\circ\lambda^{-1}_{r_g}]=[\mu_X\circ |\mathcal{K}(g)|\circ\lambda^{-1}_{r_1}]
	\end{equation}
	 in $\pi_n(X, x_0)$, the classical $n$-th homotopy group of $X$. For an arbitrary partition $T$ of $L^n$,
	there is a sequence of partitions
	\begin{equation}\label{eqq3}
	T=A_d\supseteq\cdots\supseteq A_{k}\supseteq A_{k-1}\supseteq\cdots\supseteq A_1\supseteq A_0= T_1,
	\end{equation}
	the former one being a refinement of the latter one  by adding an $(n-1)$-dimensional hyperplane.  For each $k$,  
	construct a continuous  map $$F_k: (L^n, \partial L^n)\rightarrow (X, x_0)$$ by
	\begin{equation*}
	F_{k}(w)=\mu_X\circ|\mathcal{K}(g)|\circ\lambda^{-1}_{A_{k}}(w)
	\end{equation*}
	where $0\leq i\leq d$.
	If we can show that
	\begin{equation}\label{eq6}
	[F_{k}]=[F_{k-1}],
	\end{equation}
	then this will imply
	\begin{equation}\label{eqq4}
	[F_{d}]=\cdots=[F_0]=[\mu_X\circ|\mathcal{K}(g)|\circ\lambda^{-1}_{1}].
	\end{equation}
	
	Denote by $$-1=a_{-m^{\prime}}<a_{-m^{\prime}+1}<\cdots<a_0<\cdots <a_m=1$$ the $1$-st component of the partition $A_{k-1}$.
	Without loss of generality we may assume the $(n-1)$-dimensional hyperplane added to $A_{k}$ from $A_{k-1}$ is $x_1=\tilde a_{0}$, where $a_0<\tilde{a}_0<a_1$. 
	Let $$S: (L^n, \partial L^n)=\left(\lambda_{A_k}(|\mathcal{K}(\gamma_{A_k})|), \partial \lambda_{A_k}(|\mathcal{K}(\gamma_{A_k})|)\right)\rightarrow (X, x_0)$$ be the continuous map defined by
	\begin{align*}
	(x_1, x_2,\cdots, x_n)\mapsto\begin{cases}
	F_{k}(a_0, x_2,\cdots, x_n),~\text{if}~x_1\in [a_0, \tilde a_{0}],\\
	F_{k}(x_1, x_2,\cdots, x_n),~\text{if}~x_1<a_0,\\
F_{k}\lvert [a_{0}, \tilde a_{0}],~\text{if}~x_1\in [\tilde a_{0}, a_{1}],\\
F_{k}\lvert [\tilde a_{0},  a_{1}],~\text{if}~x_1\in [ a_{1}, a_{2}],\\
F_{k}\lvert [a_{i-1},  a_{i}],~\text{if}~x_1\in [ a_{i}, a_{i+1}],~i>1.
	\end{cases}
	\end{align*}
	 Checking through the definitions, one see that  
	 $ F_{k-1}$ is equal to $S$ at all vertices of $\lambda_{A_k}(|\mathcal{K}(\gamma_{A_k})|)$ 
	 except at the vertex $$(\tilde a_{0}, x_2,\cdots, x_n).$$ Then by the choice of $r_1$, we have that the following two points
	 \begin{align*}
	&b_1= (\lambda_{r_1})\circ\lambda^{-1}_{A_{k-1}}(\tilde a_{0}, x_2,\cdots, x_n),\\ &b_2=(\lambda_{r_1})\circ\lambda^{-1}_{A_{k}}(a_0, x_2,\cdots, x_n),
	 \end{align*}
	lie in $|f^{\prime}|^{-1}(|\sigma|_{\epsilon})$
for some maximal simplex $\sigma$ of $\mathcal{K}(X)$. Note that $f^{\prime}$ takes simplices to simplices.
With this in mind one can draw a stronger conclusion that $|f^{\prime}|(b_1)$ and $|f^{\prime}|(b_2)$ are indeed contained in $|\sigma|$, since $b_1$ and $b_2$ lie in some simplex of $K$,
$\mathcal{K}(\gamma_{r_1})$ being regarded as a refinement of $K$.
This means that 
$\mu_X\circ |f^{\prime}|(b_1)$ and $\mu_X\circ|f^{\prime}|(b_2)$ are comparable, all vertices of a simplex being pairwise comparable, which is equivalent to saying that $F_{k-1}$ and $S$ are comparable at the point $(\tilde a_{0}, x_2,\cdots, x_n).$ 
Hence we can learn that
\begin{equation}\label{eqq1}
[F_{k-1}]= [S]
\end{equation}

Having taken care of the comparability between $F_{k-1}$ and $S$, we turn now to show that $F_k$ and $S$ are also comparable. Let $$S_0: (L^n, \partial L^n)=\left(\lambda_{A_k}(|\mathcal{K}(\gamma_{A_k})|), \partial \lambda_{A_k}(|\mathcal{K}(\gamma_{A_k})|)\right)\rightarrow (X, x_0)$$ be the continuous map defined by
\begin{align*}
(x_1, x_2,\cdots, x_n)\mapsto\begin{cases}
F_{k}(\tilde a_0, x_2,\cdots, x_n),~\text{if}~x_1\in [\tilde a_0,  a_{1}],\\
F_{k}(x_1, x_2,\cdots, x_n),~\text{if}~x_1<\tilde a_0,\\
F_{k}\lvert [\tilde a_{0},  a_{1}],~\text{if}~x_1\in [ a_{1}, a_{2}],\\
F_{k}\lvert [a_{i-1},  a_{i}],~\text{if}~x_1\in [ a_{i}, a_{i+1}],~i>1.
\end{cases}
\end{align*}
Similarly, when $i>0$, let
$$S_i: (L^n, \partial L^n)=\left(\lambda_{A_k}(|\mathcal{K}(\gamma_{A_k})|), \partial \lambda_{A_k}(|\mathcal{K}(\gamma_{A_k})|)\right)\rightarrow (X, x_0)$$
be the continuous map given by
\begin{align*}
(x_1, x_2,\cdots, x_n)\mapsto\begin{cases}
F_{k}( a_i, x_2,\cdots, x_n),~\text{if}~x_1\in [ a_i,  a_{i+1}],\\
F_{k}(x_1, x_2,\cdots, x_n),~\text{if}~x_1< a_i,\\
F_{k}\lvert [a_{i},  a_{i+1}],~\text{if}~x_1\in [ a_{i+1}, a_{i+2}].
\end{cases}
\end{align*}
Denote $S_{-1}=S$. When $-1\leq i\leq m-1$, consider the relation between $S_i$ and $S_{i+1}$.
The proof that $F_{k-1}$ and $S_{-1}$ is comparable still works in this context, so
\begin{equation}\label{eqq2}
[S]=[S_{-1}]=[S_0]=\cdots=[S_m]=[F_k].
\end{equation}
In view of (\ref{eqq1}) and (\ref{eqq2}), the proof of (\ref{eq6}) is finished. 

Now let $T=r_1$ and $T=T_{r_g}$ in (\ref{eqq3}) respectively. Then by (\ref{eqq4})
$$ [\mu_X\circ|\mathcal{K}(g)|\circ\lambda^{-1}_{r_1}]=[\mu_X\circ|\mathcal{K}(g)|\circ\lambda^{-1}_{1}]=[\mu_X\circ|\mathcal{K}(g)|\circ\lambda^{-1}_{r_g}],$$
 which finishes the proof of the claim (\ref{surj}).
Note that 
$$[\mu_X\circ|f^{\prime}|]=[\mu_X\circ |\mathcal{K}(g)|\circ\lambda^{-1}_{r_1}].$$
 Since $\mu_X$ is a weak homotopy equivalence {\rm{\cite[Theorem 1.4.6]{JAB2011}}}, the equality (\ref{surj}) leads us to conclude that $$\theta(g)=[|f^{\prime}|]=[f].$$ 
The proof of the surjectivity of $\theta$ is complete.

For injectivity of $\theta$, suppose $h: (\mathbb{Z}^n, \partial\mathbb{Z}^n)\rightarrow (X, x_0)$ is a map of radius $r_h$ such that $\theta([h])=0$. 
Let $G: (L^n, \partial L^n)\times I\rightarrow (|\mathcal{K}(X)|, x_0)$ is a homotopy from 
$|\mathcal{K}(h)|\circ\lambda_{r_h}^{-1}$ to the constant map $C_{x_0}$. Since $L^n\times I$ can be viewed as a subspace of $L^{n+1}$, the preceding argument can be applied to the map $G$ with $G$ in place of $f$ in the current situation. Thus we can obtain a simplicial map $G^{\prime}$ and a homotopy $H$ corresponding to $f^{\prime}$ and $g$ respectively.

The homotopy $H$  takes the form
\begin{align*}
&H: (\mathbb{Z}^n, \partial\mathbb{Z}^n)\times I_{p}\rightarrow (X, x_0),\\
&(w,t)\mapsto\begin{cases}
{\widetilde \mu_X}\circ J_{T_{r_1}}(|G^{\prime}|)(w, t),~\text{if}~(w, t)\in \gamma_{{r_1}}\times I_p,\\
x_0,~\text{otherwise,}
\end{cases}
\end{align*} 
where $p\geq 1$. 
Notice that  $|\mathcal{K}(h)|\circ\lambda_{r_h}^{-1}$ is already simplicial on the underlying simplicial complex of $\lambda_{r_h}(|\mathcal{K}(\gamma_{r_h})|)$. Then by the proof in simplicial approximation theorem {\rm{\cite[Theorem 16.1]{JRM1984}}}, $G^{\prime}$ can be chosen so that the relationship between $H_0$ and $h$ is essentially the same as the relationship between $F_k$ and $S$ in the proof for the surjectivity part. Thus the similar argument shows that 
$$[H_0]=[h].$$
 Then $[h]=0$ in $\pi_n^D(X, x_0)$, proving the injectivity of $\theta$. 
\end{proof}

 To get a feeling for how the discrete homotopy groups can be derived combinatorially, let us first look at a simple example.

\begin{Example}\label{trivial homotopy}
{\rm If $X$ is a poset with a maximum or minimum $x_0$, then all discrete homotopy groups vanish. To see this, note that any map $f: \mathbb{Z}^n\rightarrow X$ is comparable with the constant map $C_{x_0}: \mathbb{Z}^n\rightarrow x_0$. More precisely, either $f\leq C_{x_0}$ or $f\geq C_{x_0}$ occurs. For example, the poset $X$ depicted below has a maximum $d$ and therefore $\pi_n^{D}(X, d)=0$ for all $n$.
\begin{equation*}
\xymatrix{ & d&\\
		 b\ar[ur]& &  c\ar[ul]\\
	a\ar[u]&	&.
}
\end{equation*}
}

\end{Example}

 We now directly compute the $1$-st discrete homotopy group of a circle in the next example, leading to a new viewpoint towards the fundamental group of a circle.

\begin{Example}\label{circle}
	{\rm Let $X$ be a  poset described below with a fixed basepoint $b$, which is a finite model of the $1$-dimensional sphere $S^1$. 
		\begin{equation*}
		\xymatrix{   c &  d\\
			a\ar[u] \ar[ur]	&b \ar[u] \ar[ul].
		}
		\end{equation*}
		Let $f\in\pi^D_1(X, b)$ be a map  with radius $r$.
		Then $f$ can be represented by
		\begin{equation*}\label{D11}
		\cdots b\rightarrow x_{-r+1}\leftarrow\cdots x_{-1}\leftarrow x_0\rightarrow x_1\cdots\rightarrow x_{r-1}\leftarrow b\cdots,
		\end{equation*}
		where $x_i$ is the value of $f$ at $i$. 
		If two consecutive points $x_i$ and $x_{i+1}$ are equal, $-r+1\leq i\leq r-2$,  for example,  $x_i=x_{i+1}$ for $0\leq i\leq r-2$ and $i$ is even, 
		then $f\geq f_1$, where $f_1\in\pi^D_1(X, b)$ is defined by
		\begin{equation*}
		\cdots b\rightarrow \cdots \leftarrow x_0\cdots x_{i-1}\leftarrow x_i\rightarrow x_{i+2}\leftarrow x_{i+2}\rightarrow x_{i+3}\cdots x_{r-1}\leftarrow b\cdots.
		\end{equation*}
		Repetition of this process eventually produces a map $g\in \pi^D_1(X, b)$ given by
		\begin{equation}\label{D12}
		\cdots b\rightarrow y_{-n+1}\leftarrow\cdots y_{-1}\leftarrow y_0\rightarrow y_1\cdots\rightarrow y_{m-1}\leftarrow b\cdots,
		\end{equation}
		where $m, n<r$, $y_{-n+1},~ y_{m-1}\not=b$, and two consecutive points $y_i$ and $y_{i+1}$, $-n\leq i\leq m-2$,  are not equal.
		
		Denote by $C_b$ the constant map $\mathbb{Z}\rightarrow b$. 
		If  $f(i)\not= a$ for all $i$, then $$ C_b\leq f,$$ which means $[f]=0\in \pi^D_1(X, b)$.  
		We now assume  the values of $f$, thus $g$, contain $a$. If $y_i=y_{i+2}=a$, then 
		\begin{align}
		g&= \cdots a\rightarrow y_{i+1}\leftarrow a\rightarrow y_{i+3}\cdots\nonumber\\
		> g_1&=\cdots a\rightarrow a\leftarrow a\rightarrow y_{i+3}\cdots\nonumber\\
		< g_2&=\cdots a\rightarrow y_{i+3}\leftarrow y_{i+3}\rightarrow y_{i+3}\cdots\nonumber\\
		> g_3&=\cdots a\rightarrow y_{i+3}\leftarrow b\rightarrow y_{i+3}\cdots\nonumber.
		\end{align}
		This argument does apply for the case that $y_i=y_{i+2}=b$. 
		After repeating this process, we obtain a map $h\in \pi^D_1(X, b)$ given by
		\begin{equation}\label{D122}
		\cdots b\rightarrow z_{-n+1}\leftarrow\cdots z_{-1}\leftarrow z_0\rightarrow z_1\cdots\rightarrow z_{m-1}\leftarrow b\cdots,
		\end{equation}
		where if $z_i=a$ for $-n+2\leq i\leq m-2$, then $z_{i-2}=z_{i+2}=b$, and if $z_i=b$ for $-n+4\leq i\leq m-4$, then $z_{i-2}=z_{i+2}$, and two consecutive vertices $z_{i}$ and $z_{i+1}$, $-n+1\leq i\leq m-2$, are distinct.
		
		Note that $b$ appears in (\ref{D122}) with a period of four in the form
		\begin{equation}\label{D111}
		b\rightarrow z_{i}\leftarrow a\rightarrow z_{i+2}\leftarrow b,
		\end{equation} 
		where $-n+1\leq i\leq m-3$.
		If $z_i$ and $z_{i-1}$ are equal, for example, $z_i=z_{i+2}=c$, then
		\begin{align}
		h&= \cdots  b\rightarrow c\leftarrow a\rightarrow c\leftarrow b\cdots\nonumber\\
		< h_1&=\cdots  b\rightarrow c\leftarrow c\rightarrow c\leftarrow b\cdots\nonumber\\
		>h_2&=\cdots  b\rightarrow b\leftarrow b\rightarrow b\leftarrow b\cdots\nonumber.
		\end{align}
		This process, together with the methods of obtaining (\ref{D12}), can be repeated to get a map $k\in \pi^D_1(X, b)$ defined by
		\begin{equation}\label{D1222}
		\cdots b\rightarrow v_{-n+1}\leftarrow\cdots v_{-1}\leftarrow v_0\rightarrow v_1\cdots\rightarrow v_{m-1}\leftarrow b\cdots,
		\end{equation}
		in which (\ref{D111}) becomes  one of the two forms:
		\begin{align*}
		&b\rightarrow c\leftarrow a\rightarrow d\leftarrow b,\\
		&b\rightarrow d\leftarrow a\rightarrow c\leftarrow b,
		\end{align*}
		the numbers of them in (\ref{D1222}) being denoted as $t_1$ and $t_2$, respectively.
		
		Let $e\in \pi^D_1(X, b)$ be a map of radius $2$ given by
		\begin{equation*}
		\cdots b\rightarrow d\leftarrow a\rightarrow c\leftarrow b\cdots.
		\end{equation*}
		Observe that
		\begin{align*}
		e\cdot e^{-1}=&\cdots  b\rightarrow d\leftarrow a\rightarrow c\leftarrow b\rightarrow c\leftarrow a\rightarrow d\leftarrow b\cdots\\
		<&\cdots  b\rightarrow d\leftarrow a\rightarrow c\leftarrow c\rightarrow c\leftarrow a\rightarrow d\leftarrow b\cdots\\
		>&\cdots  b\rightarrow d\leftarrow a\rightarrow a\leftarrow a\rightarrow a\leftarrow a\rightarrow d\leftarrow b\cdots\\
		<&\cdots  b\rightarrow d\leftarrow d\rightarrow d\leftarrow d\rightarrow d\leftarrow a\rightarrow d\leftarrow b\cdots\\
		>&\cdots  b\rightarrow b\leftarrow b\rightarrow b\leftarrow b\rightarrow b\leftarrow b\rightarrow b\leftarrow b\cdots=C_b.
		\end{align*}
		Similar argument implies $e^{-1}\cdot e=C_b$.
		We have just seen that  $[f]=[k]$ is the product of finitely many $e$ and $e^{-1}$. With the techniques of getting (\ref{D12}),
		$k$ is homotopic to a map $l\in \pi^D_1(X, b)$, which is an element of the free group $F([e])$ generated by $[e]$.
		This says $\pi^D_1(X, b)=F([e])=\mathbb{Z}$. In particular, $$[f]=[k]=(t_2-t_1)[e].$$
	}
\end{Example}

\begin{Remark}
	{\rm The computations for classical homotopy groups for posets rely heavily on the topological structures of  spheres of arbitrary dimensions and their associated order complexes, and the continuous maps between them. It is known that to actually compute the classical homotopy groups is usually too difficult a procedure to be carried out in practice.
		
	The process of obtaining the discrete homotopy groups of posets differs from the classical one in a significant way, as the last two examples show. The computation of these groups proceeds via the classification of all comparable maps in the category of posets, rather than via the homotopy classification of continuous maps in the category of topological spaces. This fact  obviates the need for their associated order complexes. Problems regarding homotopy invariants of posets can now be tackled in a more visual and direct way in their own context. 
		}
\end{Remark}

\section{Discrete  homology of finite posets}\label{section4}
\subsection{Discrete cubical homology}
Let $X$ be a finite poset. Denote by $Q_1$ the finite poset with vertex set $\{0, 1\}$ and an edge  $(0, 1)$. For each $n\geq 1$, let $Q_n$ be the finite poset given by
\begin{equation*}
Q_n=\begin{cases}
0 ~&\text{if}~n=0,\\
Q_1^n~&\text{if}~n\geq 1.
\end{cases}
\end{equation*}
\begin{Definition}
\rm{An\emph{ $n$-cube} of $X$ is a map $\sigma: Q_n\rightarrow X$.}
\end{Definition}
Let $L_n(X)$ be the free abelian group generated by all  $n$-cubes.
 We define two face maps from $L_n(X)$ to $L_{n-1}(X)$. Let  $\sigma\in L_n(X)$. For $n\geq 1$, $i\in\{1,\cdots,n\}$ and $(a_1,\cdots, a_{n-1})\in Q_{n-1}$, let $f_i^{-}\sigma$ and $f_i^{+}\sigma$ be the $(n-1)$-cubes of $X$ given by
 \begin{align*}
&f^{-}_i\sigma(a_1,\cdots, a_{n-1})=\sigma(a_1,\cdots, a_{i-1}, 0, a_i,\cdots, a_{n-1})\\
 &f^{+}_i\sigma(a_1,\cdots, a_{n-1})=\sigma(a_1,\cdots, a_{i-1}, 1, a_i,\cdots, a_{n-1}).
 \end{align*}
 
An $n$-cube $\sigma$ is \emph{degenerate} if $f_i^-=f_i^+$ for some $i$. Let $D_n(X)$ be the free abelian group generated by all degenerate $n$-cubes and $C_n(X)=L_n(X)/D_n(X)$, whose elements are called $n$-chains in $X$.
\begin{Definition}
\rm{The \emph{boundary} of an $n$-cube $\sigma: Q_n\rightarrow X$ is
\begin{equation*}
\partial_n^{\rm Cube}(\sigma)=\sum_{i=1}^n(-1)^i(f_i^-\sigma-f_i^+\sigma).
\end{equation*}
}
\end{Definition}
 The boundary operator extends to a group homomorphism $\partial_n^{\rm Cube}: L_n(X)\rightarrow L_{n-1}(X)$. Since $\partial_n^{\rm Cube}(D_n(X))\subseteq D_{n-1}(X)$, we obtain a map $\partial_n^{\rm Cube}: C_n(X)\rightarrow C_{n-1}(X)$ and a chain complex $(C_*(X), \partial_*^{\rm Cube})$.
Let $$H^{\rm Cube}_n(X)=\rm{Ker}\partial_n^{\rm Cube}/\rm{Im}\partial_{n+1}^{\rm Cube}.$$
The group $H_n^{\rm Cube}(X)$ is called the \emph{$n$-th discrete  cubical homology group} of $X$. If $z\in \rm{Ker}\partial_n^{\rm Cube}$, then we write $[z]$ for the equivalence class of $z$ in $H_n^{\rm Cube}(X)$.

The connection between the discrete homotopy theory and the discrete homology theory is expressed by the following theorem.
\begin{Theorem}
Let $X_1$ and $X_2$ be finite posets. If $X_1$ and $X_2$ are homotopy equivalent, then 
$$H^{\rm Cube}_n(X_1)=H^{\rm Cube}_n(X_2)$$ for all $n\geq 0$.
\end{Theorem}
The idea of the proof of this theorem is exacly the same as the one in {\rm{\cite[Theorem 3.4]{HB2019}}}. We will not prove this theorem here.

We now construct a natural homomorphism from the discrete to the simplicial homology groups.

Let $K$ be a simplicial complex. For $n\geq 0$, we denote by $L^{\rm Simpl}_n(K)$ the free abelian group generated by all nondegenerate simplices of $K$ and write $D^{\rm Simpl}_n(K)$ for the free abelian group generated by all degenerate simplices of $K$. 
Let $$C^{\rm Simpl}_n(K)=L^{\rm Simpl}_n(K)/D^{\rm Simpl}_n(K).$$
The $n$-th simplicial homology group of $K$ is denoted  $H_n^{\rm Simpl}(K)$ whose associated chain complex is
$(C^{\rm Simpl}_*(K), \partial_*^{\rm Simpl})$. 
Similarly, $H_*^{\rm Sing}$  denotes the singular homology theory. Denote by $\eta_*$ the canonical homomorphism $$H_n^{\rm Simpl}(K)\rightarrow H_*^{\rm Sing}(|K|).$$

To establish a natural homomorphism between the homology groups ${H}_*^{\rm Cube}(X)$ and ${H}_*^{\rm Simpl}(\mathcal{K}(X))$,  we are motivated by the construction of the one in {\rm{\cite{HB2019}}}, which is a natural  map from the cubical homology groups to the path homology groups. This map will be defined in a similar fashion.

 Consider an $n$-cube $\sigma: Q_n\rightarrow X$ with $n\geq 1$. In order to define a map from ${C}_n^{\rm Cube}(X)$ to ${C}^{\rm Simpl}_n(\mathcal{K}(X))$, we associate to any permutation $\tau\in S_n$ an $n$-chain $p_{\tau}=\{x_0,x_1,\cdots,x_n~|~x_0<x_1<\cdots<x_n\}$ from $(0,\cdots, 0)\in V(Q_n)$ to $(1,\cdots, 1)\in V(Q_n)$. The chain $p_{\tau}$ is defined as 
follows. Let $x_0=(0,\cdots,0)$. For $0<i\leq n$,
flipping the $\tau(i)$-th coordinate of $x_{i-1}$ from $0$ to $1$  gives $x_{i}$. 
We write $p_{\tau}(i)$ for the $i$-th vertex in the chain $p_{\tau}$ for $0\leq i\leq n$.

 To each $n$-cube $\sigma\in {L}_n^{\rm Cube}(X)$ for $n\geq 1$, we assign the chain 
\begin{align}
\psi(\sigma)&=\sum_{\tau\in S_n}{\rm sign}(\tau)\mathcal{K}(\{\sigma\circ p_{\tau}(i)~|~0\leq i\leq n\})\nonumber\\
&=\sum_{\tau\in S_n}{\rm sign}(\tau)[\sigma\circ p_{\tau}(0), \sigma\circ p_{\tau}(1),\cdots, \sigma\circ p_{\tau}(n)]\nonumber
\end{align}
of ${L}^{\rm Simpl}_n(\mathcal{K}(X))$.
It is not hard to see that if $\sigma\in{D}_n^{\rm Cube}(X)$, then $\psi(\sigma)\in
{D}^{\rm Simpl}_n(\mathcal{K}(X))$.

The proof that $\psi$ is a chain map is similar to the one in {\rm{\cite[Lemma 5.1]{HB2019}}}. For completeness, we will sketch this proof.
\begin{Proposition}
Let $\sigma\in {C}_n^{\rm Cube}(X)$. Then
$$\partial_n^{\rm Simpl}\psi(\sigma)=\psi(\partial_n^{\rm Cube}(\sigma)).$$
Hence $\psi$ is a chain map and induces a homomorphism $$\psi_*: {H}_p^{\rm Cube}(X)\rightarrow  {H}_p^{\rm Simpl}(\mathcal{K}(X))$$
for each $p\geq 0$.
\end{Proposition}
\begin{proof}
If $n=0$ or $n=1$, the proof is trivial and then in what follows we assume $n\geq 2$.
By the  boundary formula,
\begin{equation}\label{eq1}
\partial_n^{\rm Simpl}([\sigma\circ p_{\tau}(0), \cdots, \sigma\circ p_{\tau}(n)])=\sum_{i=0}^{n}(-1)^i[\sigma(p_{\tau}(0)),\cdots, \widehat{\sigma(p_{\tau}(i))},\cdots, \sigma(p_{\tau}(n))].
\end{equation}
Let $1\leq l\leq n-1.$ If $\tau^{\prime}$ is constructed from $\tau$ by interchanging
$\tau(l)$ and $\tau(l+1)$, then $p_{\tau}(i)=p_{\tau^{\prime}}(i)$ for $i\not=l$. Thus the $l$-th summands of (\ref{eq1}) for $p_{\tau}$ and $p_{\tau^{\prime}}$ coincide. In addition, ${\rm sign}(\tau)={\rm -sign}(\tau^{\prime})$. This means that
\begin{align*}
\partial_n^{\rm Simpl}\psi(\sigma)=&\sum\limits_{\tau\in S_n}{\rm sign}(\tau)\partial_n^{\rm Simpl}([\sigma\circ p_{\tau}(0), \cdots, \sigma\circ p_{\tau}(n)])\\
=&\sum\limits_{\tau\in S_n}{\rm sign}(\tau)([\sigma(p_{\tau}(1)),\cdots, \sigma(p_{\tau}(n))]  \nonumber\\ 
&+(-1)^n[\sigma(p_{\tau}(0)),\cdots, \sigma(p_{\tau}(n-1))]).\nonumber
\end{align*}

To each $\tau\in S_n$, we associate the following two permutaions.
\begin{itemize}
	\item [(1)] $\tau^{-}\in S_{n-1}$ is defined as the permutation where $\tau^{-}(j)=\tau(j)$ if $\tau(j)<\tau(n)$ and $\tau(j)-1$ if $\tau(j)>\tau(n)$.
	\item[(2)] $\tau^{+}\in S_{n-1}$ is defined as the permutation where $\tau^{+}(j)=\tau(j+1)$ if $\tau(j+1)<\tau(1)$ and $\tau(j+1)-1$ if $\tau(j+1)>\tau(1)$.
\end{itemize}
Then we have 
\begin{align*}
\psi(f_i^{-}\sigma)&=\sum\limits_{\tau\in S_n,  \tau(n)=i} {\rm sign}(\tau^{-})[(f_i^{-}\sigma)\circ( p_{\tau^{-}}(0)),\cdots, (f_i^{-}\sigma)\circ (p_{\tau^{-}}(n))]\nonumber\\
&=(-1)^{n-i}\sum\limits_{\tau\in S_n,  \tau(n)=i} {\rm sign}(\tau)[(f_i^{-}\sigma)\circ( p_{\tau^{-}}(0)),\cdots, (f_i^{-}\sigma)\circ (p_{\tau^{-}}(n))],\\
\psi(f_i^{+}\sigma)&=\sum\limits_{\tau\in S_n,  \tau(1)=i} {\rm sign}(\tau^{+})[(f_i^{+}\sigma)\circ( p_{\tau^{+}}(0)),\cdots, (f_i^{+}\sigma)\circ (p_{\tau^{+}}(n))]\nonumber\\
&=(-1)^{i-1}\sum\limits_{\tau\in S_n,  \tau(1)=i} {\rm sign}(\tau)[(f_i^{+}\sigma)\circ( p_{\tau^{+}}(0)),\cdots, (f_i^{+}\sigma)\circ (p_{\tau^{+}}(n))].
\end{align*}
Finally, observe that for $j=\tau(n)$ and $j^{\prime}=\tau(1)$,
\begin{equation}\label{eq2}
\begin{split}
\sigma(p_{\tau}(m))&=(f_j^{-}\sigma)\circ (p_{\tau^{-}}(m)),\\
\sigma(p_{\tau}(r))&=(f_{j^{\prime}}^{+}\sigma)\circ (p_{\tau^{+}}(r-1)),
\end{split}
\end{equation}
where $0\leq m\leq n-1,~ 1\leq r\leq n$.
Thus we get 
\begin{equation*}
\begin{split}
&\psi(\partial_n^{\rm Cube}(\sigma))=\sum_{i=1}^{n}(-1)^i(\psi(f_i^{-}\sigma)-\psi(f_i^+\sigma))\\
=&\sum_{i=1}^{n}(-1)^i\left((-1)^{n-i}\sum\limits_{\tau\in S_n,  \tau(n)=i} {\rm sign}(\tau)[(f_i^{-}\sigma)\circ( p_{\tau^{-}}(0)),\cdots, (f_i^{-}\sigma)\circ (p_{\tau^{-}}(n))]\right.\\
&\left.-(-1)^{i-1}\sum\limits_{\tau\in S_n,  \tau(1)=i} {\rm sign}(\tau)[(f_i^{+}\sigma)\circ( p_{\tau^{+}}(0)),\cdots, (f_i^{+}\sigma)\circ (p_{\tau^{+}}(n))]\right)\\
=&\sum_{\tau\in S_n}{\rm sign}(\tau)\Big((-1)^n[\sigma(p_{\tau}(0)),\cdots, \sigma(p_{\tau}(n-1))]\\
&+[\sigma(p_{\tau}(1)),\cdots, \sigma(p_{\tau}(n))] \Big)~(\text{by (\ref{eq2})})\\
=&\partial_n^{\rm Simpl}\psi(\sigma).
\end{split}
\end{equation*}
\end{proof}

These two homology theories will turn out to be closely related by the homomorphism $\psi_*$ in what follows.

\begin{Proposition}\label{1h}
Let $X$ be a homogeneous poset of dimension $1$. The map $\psi$ induces an isomorphism $\psi_*: {H}_p^{\rm Cube}(X)\rightarrow  {H}_p^{\rm Simpl}(\mathcal{K}(X))$ for any $p\geq 0$.
\end{Proposition}
\begin{proof}
The case $p=0$ is trivial.
The dimension of every maximal chain of $X$ is not greater than $1$, then this says $C_p^{\rm Simpl}(X)=0$ for $p>1$, so ${H}_p^{\rm Simpl}(\mathcal{K}(X))=0$ for $p>1$. It is not difficult to see that every $p$-cube with $p>1$ is degenerate, which means $C_p^{\rm Cube}(X)=0$. Then ${H}_p^{\rm Cube}(\mathcal{K}(X))=0$ for $p>1$. 

Consider $p=1$. A nondegenerate $1$-cube can be represented by a directed edge $(v_1,v_2)$
 of the Hasse diagram of $X$, while a nondegenerate $1$-simplex is $\{v_1, v_2\}$ where $v_1$ and $v_2$ are vertices of $X$. Hence $C_1^{\rm Cube}(X)$ has the same basis as $C_1^{\rm Simpl}(X)$.  Because the differentials $\partial_1^{\rm Cube}$ and $\partial_1^{\rm Simpl}$ are indeed identical, $\psi_*$ is an isomorphism for $p=1$.
\end{proof}

When  $X$ is a  poset with $|\mathcal{K}(X)|$ being an $m$-manifold, Proposition \ref{1h} would  generalize to higher dimensions. Before proceeding further we need a useful excision property for finite posets.
\begin{Lemma}\label{retract}
Let $X$ be a finite poset and let $Y$ be a subposet of $X$. Then $|\mathcal{K}(X-Y)|$ is a strong deformation retract of $|\mathcal{K}(X)|-|{\mathcal{K}(Y)}|$. 
\end{Lemma}
The argument given in {\rm{\cite[Theorem 1.4.6]{JAB2011}}} works just as well for this case.

\begin{Theorem}\label{Main1}
	Let $X$ be a  poset with $|\mathcal{K}(X)|$ being a closed $m$-manifold.
Then the homomorphism $\psi_*: {H}_p^{\rm Cube}(X)\rightarrow  {H}_p^{\rm Simpl}(\mathcal{K}(X))$ is a surjection for $p\not= m$.
\end{Theorem}
\begin{proof}
We prove the theorem by induction on the number $n$ of generators of chain complex $\mathcal{C}_*^{\rm Cube}(X)$. If $n=1$, then $X$ consists of a single vertex $v$. In each dimension $p$, there is exactly one $p$-cube $\sigma: Q_p\rightarrow v$ of $X$ and
 exactly one ordered $p$-simplex $\{v,\cdots, v\}$ of $\mathcal{K}(X)$. Furthermore, $\psi(\sigma)=\{v,\cdots, v\}$. Then $\psi$ is an isomorphism already on the chain level. Hence 
$\psi_*$ is surjective in this case.

Now suppose the theorem holds for any finite poset  whose associated cubical chain complex has fewer than $n$ generators. Let $X$ be a poset with  ${C}_*^{\rm Cube}(X)$ having $n$ generators.
 Let $p\geq 0$. If $p>m$, then $${H}_p^{\rm Simpl}(\mathcal{K}(X))=0,$$ 
 where the equality follows from {\rm{\cite[Lemma 3.27]{AH2002}}}. Therefore $\psi_*$ is surjective for $p>m$.
 
The case $p=0$ is trivial. Assume that $0<p\leq m-1$. Let $\sigma: Q_p\rightarrow X$  be a generator of ${C}_p^{\rm Cube}(X)$. If $\sigma(Q_p)=X$, then $$H_p^{\rm Simpl}(\mathcal{K}(X))=H_p^{\rm Simpl}(\mathcal{K}(\sigma(Q_p)))=0$$ since $\sigma(Q_p)$ is contractible, so we are done. Now assume $\sigma(Q_p)\not=X$.
Consider the following commutative diagram, where $i_*, j_*$ are induced by inclusion maps.
 \begin{equation*}\label{diagram1}
\xymatrix{{H}_p^{\rm Cube}(X)\ar[r]^{\psi_*}  & {H}_p^{\rm Simpl}(|\mathcal{K}(X)|)\\
{H}^{\rm Cube}_p(X-\sigma(Q_p)) \ar[u]^{i_*} \ar[r]_-{\psi_*}  &{H}_p^{\rm Simpl}(|\mathcal{K}(X-\sigma(Q_p))|)\ar[u]^{j_*}.
}
\end{equation*}
Note that $|\mathcal{K}(X-\sigma(Q_p))|$  is an $m$-manifold, being a strong deformation retract of $|\mathcal{K}(X)|-|{\mathcal{K}(\sigma(Q_p))}|$ by Lemma \ref{retract}.
It follows from the induction hypothesis that the map $\psi_*$ on the bottom line of the diagram is surjective for $p\leq m-1$. We shall show that $j_*$ is  a surjection for $p\leq m-1$, which means that the map $\psi_*$ on the top line has the same property and the proof is complete.

Because $\sigma(Q_p)$ is contractible, ${\mathcal{K}(\sigma(Q_p))}$ is collapsible {\rm{\cite[Theorem 4.2.11]{JAB2011}}}. 
Let us take $N$ to be a regular neighbourhood of $|{\mathcal{K}(\sigma(Q_p))}|$ in $|\mathcal{K}(X)|$.  It is known that $N$ is a PL $m$-ball {\rm{\cite[Corollary 3.27]{CPR1972}}}.
By excision, we have
$${H}_*^{\rm Sing}(|\mathcal{K}(X)|, |\mathcal{K}(X)|-|{\mathcal{K}(\sigma(Q_p))}|)= {H}_*^{\rm Sing}({\rm int}N, {\rm int}N-|{\mathcal{K}(\sigma(Q_p))}|).$$
Let $h: {\rm int}N\rightarrow \mathbb{R}^m$ be a homeomorphism.
Then \begin{align*}
{H}_*^{\rm Sing}({\rm int}N, {\rm
	 int}N-|{\mathcal{K}(\sigma(Q_p))}|)&={H}_*^{\rm Sing}((\mathbb{R}^m, \mathbb{R}^m-h(|{\mathcal{K}(\sigma(Q_p))}|)),\\
{H}_*^{\rm Sing}({\rm int}N-|{\mathcal{K}(\sigma(Q_p))}|)&={H}_*^{\rm Sing}(\mathbb{R}^m-h(|{\mathcal{K}(\sigma(Q_p))}|)).
\end{align*}
If we take $B^p\subseteq\mathbb{R}^m$ to be a standard $p$-ball with radius sufficiently small, then Alexander duality gives
\begin{equation*}\label{eq4}
{H}_*^{\rm Sing}(\mathbb{R}^m-h(|{\mathcal{K}(\sigma(Q_p))}|)={H}_*^{\rm Sing}(\mathbb{R}^m-B^p).
\end{equation*}
It follows that ${H}_{i}^{\rm Sing}(|\mathcal{K}(X)|, |\mathcal{K}(X-\sigma(Q_p))|)=\mathbb{Z}$ for $i=m$ and vanishes otherwise.
  Hence
$$H_p^{\rm Sing}(|\mathcal{K}(X-\sigma(Q_p))|)=H_p^{\rm Sing}(|\mathcal{K}(X)|-|{\mathcal{K}(\sigma(Q_p))}|).$$
The long exact sequence of homology groups for the pair $(|\mathcal{K}(X)|, |\mathcal{K}(X-\sigma(Q_p))|)$ then implies $j_*$ has the desired property.
\end{proof}
In fact, we get a stronger result in this proof.
\begin{Corollary}\label{appro}
Let $p\leq m-1$. Then $\psi_*: {H}_p^{\rm Cube}(X-\sigma(Q_p))\rightarrow  {H}_p^{\rm Simpl}(\mathcal{K}(X))$ is a surjection for any $\sigma\in L_p(X)$ with $|\mathcal{K}(\sigma(Q_p))|\subseteq {\rm int}|\mathcal{K}(X)|$, where $X$ is a  poset with $|\mathcal{K}(X)|$ being an $m$-manifold.
\end{Corollary}

The following examples will help us understand the content of this theorem.
\begin{Example}\label{trivial homology}
{\rm Let $X$ be the finite poset depicted below. 
	\begin{equation*}
	\xymatrix{c\\
		b\ar[u]\\
		a\ar[u].
	}
	\end{equation*}
Let us compute  the $1$-st discrete  cubical homology group $H_1^{\rm Cube}(X)$.
Represent   $1$-cubes $\sigma: Q_1\rightarrow X$ by $(x_0, x_1)$, where $x_i=f(i)$, and represent   $2$-cubes $\sigma: Q_2\rightarrow X$ by $(x_{00}, x_{10}, x_{01}, x_{11})$, where $x_{ij}=\sigma(i, j)$. 
 It is easy to recognize $1$-cycles looking into this Hasse diagram. More precisely, every $1$-cycle must be a multiple of
$$ (a, b)+(b, c)-(a, c).$$
By direct computation,
$$\partial_2^{\rm Cube}(a, b, a, c)=(a, b)+(b, c)-(a, c),$$
which follows that 
$$H_1^{\rm Cube}(X)=0.$$
 Theorem \ref{Main1} then forces the $1$-st simplicial homology group of $\mathcal{K}(X)$ to vanish, namely,
 $$H_1^{\rm Simpl}(\mathcal{K}(X))=0.$$ 
}
\end{Example}

Theorem \ref{Main1} might suggest conjecturing that $H_*^{\rm Cube}(X)$ and $H_*^{\rm Simpl}(\mathcal{K}(X))$ are the same for all finite posets $X$. Now we construct an example showing this is not always the case.
\begin{Example}\label{con}
	{\rm
Let $X$ be the finite poset described below.
\begin{equation*}
\xymatrix{e&f\\
	c\ar[u]\ar[ur]&d\ar[u]\ar[ul] \\
	a\ar[u]\ar[ur]&b\ar[u]\ar[ul].
}
\end{equation*}
This is a finite model of the $2$-dimensional sphere, hence 
 $$H_1^{\rm Simpl}(\mathcal{K}(X))=0.$$ 
 Represent   $1$-cubes $\sigma: Q_1\rightarrow X$ by $(x_0, x_1)$, where $x_i=\sigma(i)$, and represent   $2$-cubes $\sigma: Q_2\rightarrow X$ by $(x_{00}, x_{10}, x_{01}, x_{11})$, where $x_{ij}=\sigma(i, j)$. 
 We shall prove $H_1^{\rm Cube}(X)\not=0$ by finding a $1$-cycle $\alpha\in C^{\rm Cube}_1(X)$ that is not a boundary.
 Define \begin{align*}
\alpha=(a, e)-(b, e)+(b, f)-(a, f).
\end{align*}
It is easy to check $\alpha$ is a cycle. 

If $\beta\in C_2^{\rm Cube}(X)$ satisfies $\partial_2^{\rm Cube}(\beta)\not=0$, then $\beta$ is a linear combination 
\begin{align*}
&x_1(a,c,a,e)+x_2(a,d,a,e)+x_3(a,d,c,e)\\
+&x_4(a,c,a,f)+x_5(a,d,a,f)+x_6(a,d,c,f)\\
+&x_7(b,c,b,e)+x_8(b,d,b,e)+x_9(b,d,c,e)\\
+&x_{10}(b,c,b,f)+x_{11}(b,d,b,f)+x_{12}(b,d,c,f).
\end{align*}
If $\alpha$ is a boundary, there is a $\beta\in C_2^{\rm Cube}(X)$ such that $
\partial_2^{\rm Cube}(\beta)=\alpha.$ Then we have
\begin{align*}
\begin{cases}
x_1-x_3+x_7-x_9=0\\
x_1-x_3+x_4-x_6=0\\
-x_1-x_2=1\\
x_2+x_3+x_8+x_9=0\\
x_2+x_3+x_5+x_6=0\\
x_4-x_6+x_{10}-x_{12}=0\\
-x_4-x_5=-1\\
x_5+x_6+x_{11}+x_{12}=0\\
x_7-x_9+x_{10}-x_{12}=0\\
-x_7-x_8=-1\\
x_8+x_9+x_{11}+x_{12}=0\\
-x_{10}-x_{11}=1.
\end{cases}
\end{align*}
Observe that \begin{align*}
&(-x_{10}-x_{11})=\\
&-(\left(x_7-x_9+x_{10}-x_{12})+(-x_7-x_8)+(x_8+x_9+x_{11}+x_{12})\right).
\end{align*}
 It follows that $\beta=0$, a contradiction.
}
\end{Example}

\begin{Remark}
	{\rm  It sometimes happens that computing discrete homology groups for a poset has the effect of simplifying the  computation of the classical homology groups. This is exactly what Example \ref{trivial homology} tells us. Furthermore, compared with the classical homology groups for a poset, the discrete ones sometimes have the advantage of collecting more imformation about this poset with nontrivial values. 

}
\end{Remark}

\subsection{An application of Theorem \ref{Main1}}
In this subsection, we demonstrate how Theorem \ref{Main1} can be applied to a criterion for determining the contractibility of triangulable manifolds.

Let $X$ be a homogeneous poset of dimension $m$.
  Define $X^{=p}$ as the subposet of elements of degree equal to $p$, 
$$X^{=p}=\{x^p_1,\cdots, x^p_{i_p}|~x^p_{i_j}\in X,  {\rm deg}(x^p_{i_j}) =p\}.$$
Label the vertices of $X$ by 
\begin{equation}\label{A1}
X=\bigcup_{\substack{1\leq t\leq i_{p}\\ 0\leq p\leq m}}x^{p}_{t}.
\end{equation}
It is no difficult to see that if $X$ is a poset with $|\mathcal{K}(X)|$ being an $m$-manifold, then $X$ is homogeneous.

Let $A=\mathcal{K}(\{x^{0}_{t_{0}},\cdots, x^{m}_{t_{m}}\})$ be an $m$-simplex of $\mathcal{K}(X)$ where $1\leq t_j\leq i_j,$ and let
$I^{m}=[0, 1]^{m}$ be the $m$-cube in $\mathbb{R}^{m}$. The last coordinate of $I^{m}$ is singled out for special reference and written as $I^{m}=I^{m-1}\times I$. 
Regard $|A|$ as a subpolyhedron of $I^m$ by the canonical homeomorphism $f$ given by
\begin{align*}
f(x^j_{t_j})=\begin{cases}
(0, 0\cdots, 0),&~\text{if}~j=0,\\
e_{j},&~\text{if}~ 1\leq j\leq m-1,\\
(\frac{1}{m}, \frac{1}{m}\cdots, \frac{1}{m}),&~\text{if}~j=m.
\end{cases}
\end{align*}
If $(x, t),~(x^{\prime}, t^{\prime})\in |A|\subseteq I^{m-1}\times I$, we say $(x, t)$ is \emph{directly below} $(x^{\prime}, t^{\prime})$ in $|A|$ if $x=x^{\prime}$ and $t<t^{\prime}$.

\begin{Definition}{\rm
Let $a, ~a^{\prime}\in|\mathcal{K}(X)|$. We say $a$ is in the \emph{shadow} of $a^{\prime}$ in $|\mathcal{K}(X)|$ if there is an $m$-chain $A$ of $X$ such that $|\mathcal{K}(A)|$ contains both $a$ and $a^{\prime}$, and $a$ is directly below $a^{\prime}$ in $|\mathcal{K}(A)|$.}
\end{Definition}

\begin{Definition}
	\rm{ Let $P$ be a subpolyhedron of $|\mathcal{K}(X)|$. The \emph{shadow} ${\rm sh}(P)$ of $X$ is the closure of the set of points of $P$ that are directly below some other point of $P$.} 
\end{Definition}

Let $P_2\subseteq P_1\subseteq |\mathcal{K}(X)|$ be subpolyhedra with $P_1\searrow P_2$. If $P_1\searrow P_2$ is elementary, that is $P_1\seRightarrow P_2$, then we call this elementary collapse \emph{sunny} if no points of $P_1-P_2$ are in  ${\rm sh}(P_1)$. We call a sequence of elementary sunny collapses a \emph{sunny collapse}.
 If $P_2$ is a point, then $P_1$ is called \emph{sunny collapsible}.

To get a feeling for what a sunny collapse is about, let us look at the following example.
\begin{Example}
{\rm
Let $X=\{x_1^0, x_2^0, x_1^1, x_2^1, x_1^2\}$ be a homogeneous poset of dimension $2$. Figure \ref{sunny1} shows its Hasse diagram and the geometric realization $|\mathcal{K}(X)|$.	Consider the polyhedra $P_i\subseteq |\mathcal{K}(X)|$, $1\leq i\leq3$, defined as follows:
\begin{align*}
&P_1=|\mathcal{K}(\{x_1^0, x_1^1, x_2^1, x_1^2\})|,\\
&P_2=|\mathcal{K}(\{x_1^0, x_1^1,  x_1^2\})|,\\
&P_3=|\mathcal{K}(\{x_1^0, x_1^1\}\cup \{x_1^1, x_1^2\})|.
\end{align*}
Then $P_1\seRightarrow P_2$ is a sunny collapse (see Figure \ref{sunny2}), while $P_2\seRightarrow P_3$ is not a sunny collapse (see Figure \ref{sunny3}).

\begin{figure}[htbp]
	\subfigure
	{\begin{minipage}{6cm}
			\centering
			\includegraphics[scale=0.55]{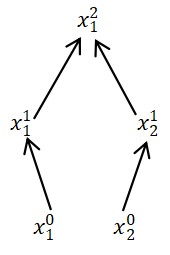}
	\end{minipage}}
	\subfigure
	{\begin{minipage}{6cm}
			\centering
			\includegraphics[scale=0.55]{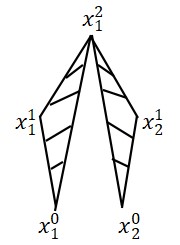}
	\end{minipage}}
	\caption{The poset $X$ and the geometric realization $|\mathcal{K}(X)|$.}
	\label{sunny1}
\end{figure}

\begin{figure}[htbp]
	\centering
	\begin{minipage}[t]{0.48\textwidth}
		\centering
		\includegraphics[width=5cm]{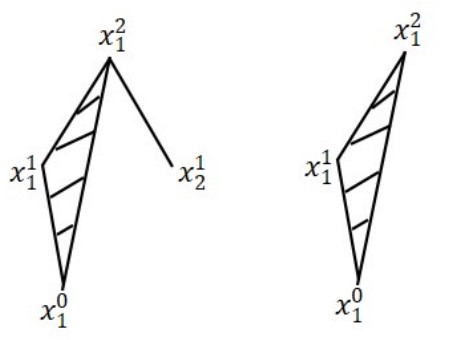}
		\caption{$P_1\seRightarrow P_2$ is a sunny collapse.}
		\label{sunny2}
	\end{minipage}
	\begin{minipage}[t]{0.48\textwidth}
		\centering
		\includegraphics[width=4cm]{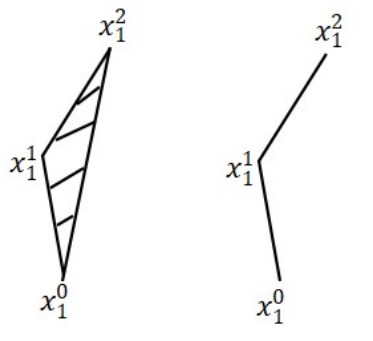}
		\caption{$P_2\seRightarrow P_3$ is not a sunny collapse.}
		\label{sunny3}
	\end{minipage}
\end{figure}
}
\end{Example}
Sunny collapses will turn out to be quite useful for proving our main theorem, though at first glance this property may seem just an idle technicality.
The following lemma shows that there exists a class of posets which are sunny collapsible.
Moreover, the fact that these posets are sunny collapsible becomes important later at key points in the proof of Lemma \ref{ballpair}, which plays a fundamental role for obtaining our main theorem.
\begin{Lemma}\label{horizontal collapsible}
	Let $X$ be a  poset with $|\mathcal{K}(X)|$ being an $m$-manifold  and $\sigma\in L_n(X)$. If ${\dim}\sigma(Q_n)<m$, then $|\mathcal{K}(\sigma(Q_n))|$ is sunny collapsible.
\end{Lemma}
\begin{proof}
	Denote by $v$ the image of $(0, \cdots, 0)$ under $\mathcal{K}(\sigma)$.
	We shall construct inductively a decreasing sequence of subpolyhedra 
	$$ |\mathcal{K}(\sigma(Q_n))|=P_0\supseteq P_1\supseteq\cdots\supseteq P_k=v,$$
	satisfying the two conditions:
	\begin{itemize}
		\item [(1)] $P_i=|vK^{n-i-1}|$, where $K^{n-i-1}$ is a homogeneous $(n-i-1)$-complex.
		\item[(2)] There is a sunny collapse $P_{i-1}\seRightarrow P_i$.
	\end{itemize}
	Let $K^{n-1}=\{A\in\mathcal{K}(\sigma(Q_n)) ~|~v~ \text{is not a vertex of}~ A\}$. Then $P_0=|vK^{n-1}|$.  
	
	Suppose we are given $P_{i-1},~ K^{n-i}$ satisfying the two conditions, we have to construct $P_i, ~K^{n-i-1}$ and verify the two conditions.
	Define 
	$$ K^{n-i-1}=\text{the}~ (n-i-1)~\text{- skeleton of}~ K^{n-i},$$
	which is a homogeneous complex since $K^{n-i}$ is homogeneous.  
	Let $P_i=|vK^{n-i-1}|$. 
	
	For each $(n-i)$-simplex $D\in K^{n-i}$, the closure ${\rm cl}(v|D|)$ is a ball and $|D|$ is a face. Collapsing each ${\rm cl}(v|D|)$ from ${\rm cl}(|D|)$, we obtain  $$|vK^{n-i}|\seRightarrow |vK^{n-i-1}|,$$ which is a sunny collapse because no points of ${\rm sh}(P_{i-1})$ are removed.
\end{proof}

{\rm{\cite[Lemma 20]{ECZ1963}}} states that if $(B^m, B^q)$ is an $(m, q)$-ball pair of codimension $m-q\geq3$, then $B^m\searrow B^q$. If $X$ is a contractible  poset with $|\mathcal{K}(X)|$ being an $m$-manifold,
then $|\mathcal{K}(X)|$ is a PL ball since a collapsible PL manifold is a PL ball {\rm{\cite[Corollary 3.27]{CPR1972}}}. Note that $$(|\mathcal{K}(X)|, |\mathcal{K}(\sigma(Q_n))|)$$ may be not a ball pair {\rm{\cite{ECZ1963}}}, since $\partial |\mathcal{K}(\sigma(Q_n))|$ and ${\rm int}|\mathcal{K}(\sigma(Q_n))|$ are not necessarily contained in $\partial |\mathcal{K}(X)|$ and ${\rm int}|\mathcal{K}(X)|$, respectively. Moreover, the codimension may be less than $3$. 
Thus we cannot apply {\rm{\cite[Lemma 20]{ECZ1963}}} to show that $$|\mathcal{K}(X)|\searrow |\mathcal{K}(\sigma(Q_n))|$$ if ${\dim}\sigma(Q_n)<m$. However, the techniques thereof can be used to obtain this result.

\begin{Lemma}\label{ballpair}
	Let $X$ be a contractible  poset with $|\mathcal{K}(X)|$ being an $m$-manifold and $\sigma\in L_n(X)$. Then 
	$|\mathcal{K}(X)|\searrow |\mathcal{K}(\sigma(Q_n))|$ if ${\dim}\sigma(Q_n)<m$.
\end{Lemma}
\begin{proof}
	We write the sunny collapse in the proof of Lemma \ref{horizontal collapsible}  as
	$$|\mathcal{K}(\sigma(Q_n))|=P_0\seRightarrow P_1\seRightarrow\cdots\seRightarrow P_r=\sigma(0,\cdots, 0).$$
	Let $W_i$ be the polyhedron consisting of $P_0$ and all points below $P_i$ in $|\mathcal{K}(X)|$. We will prove that \begin{equation*}\label{colla}
	|\mathcal{K}(X)|\searrow W_0\searrow W_1\searrow\cdots\searrow W_r\searrow P_0.
	\end{equation*}
	
	Choose a triangulation $R$ of $|\mathcal{K}(X)|$ containing a subdivision $R_0$ of $W_0$. The collapsibility of $|\mathcal{K}(X)|$ allows us to collapse $R$  to $R_0$ simplicially from the top, in order of decreasing dimension of simplices, completing the first step. The last step is obvious. If ${\rm deg}\sigma(0,\cdots, 0)<m$, then $W_r=P_0$. If ${\rm deg}\sigma(0,\cdots, 0)=m$, then  $W_r$ consists of $P_0$ joined by some single arcs formed by points below $\sigma(0,\cdots, 0)$ in $|\mathcal{K}(X)|$. Collapsing these arcs  finishes this step.
	
	Fix $i, 1\leq i\leq r$. We shall prove $W_{i-1}\searrow W_i$. If $P$ is a polyhedron of $|\mathcal{K}(X)|$, we use notation $E(P)$ to denote the polyhedron of $P$ and all points below $P$ in $|\mathcal{K}(X)|$.
	Suppose $P_{i-1}\seRightarrow P_i$ collapses $A$ from $B$, where $A=aB$.  Let $k={\rm dim} A$.  Then $A$ is of the form $$x^{p}_{t}|\mathcal{K}\{x^{p_1}_{t_{1}},\cdots,x^{p_{k}}_{t_{k}}\}|,$$ where $a=x_t^p, ~p_1<\dots<p_{k}, ~1\leq t_{j}\leq i_{p_j}$, and $B= |\mathcal{K}\{x^{p_1}_{t_{1}},\cdots,x^{p_{k}}_{t_{k}}\}|$. 
	If $$E(B)=E(A)=\emptyset,$$  we are done. 
	
	Now assume that $$E(B)\not=\emptyset,$$ this is equivalent to saying that $p_k=m$. It follows that
	$$ E(b)\not=\emptyset,$$
	for any $b\in {\rm int}B$. In particular, $E(\widehat{B})\not=\emptyset$, where $\widehat{B}$ is the barycenter of $B$.
	Let $C_1,\cdots, C_r$ be all $m$-chains of $X$ containing
	$\{x^{p_1}_{t_{1}},\cdots,x^{p_{k}}_{t_{k}}, x^p_t\}$.  Let us take  a point $c_j\in|\mathcal{K}(C_j-x^{p_{k}}_{t_{k}})|$ for each $j$  that is directly below the barycenter of $B$. If $c_j\in P_0$, then
	 $$|\mathcal{K}(C_j-x^{p_{k}}_{t_{k}})|\subseteq P_0,$$
	 since $c_j$ is by definition in the interior of $|\mathcal{K}(C_j-x^{p_{k}}_{t_{k}})|$. This forces
	 $$x^{p_{k}}_{t_{k}}|\mathcal{K}(C_j-x^{p_{k}}_{t_{k}})|\subseteq P_0,$$ $x^{p_{k}}_{t_{k}}$ being a vertex of $\sigma(Q_n)$,
	 which contradicts the dimension of $P_0$.
	 Hence  $c_j\not\in P_0$ and  $$c_jA\cap P_0=A.$$
	Let $T_j=W_i\cup c_jA$. Then $$W_{i-1}=\bigcup\limits_{j}T_j.$$
	No points of $P_i$ are above $A$ since this collapse is sunny. It follows that $W_i\cap c_jA=|x^p_t\partial(c_j B)|$, $W_i\cup c_jA=T_j$. Each $T_j\searrow W_i$ is obtained by collapsing $c_jA$  from $c_jB$. This finishes the proof that $W_{i-1}\searrow W_i$ in this case.
	
	It remains to consider the case $$E(B)=\emptyset,~E(A)\not=\emptyset.$$  This says $p=m,~ p_k<m$. We take a point $c^{\prime}_j\in|\mathcal{K}(C_j-x_t^p)|$  for each $j$  that is directly below $x_t^p$.  The argument for the previous case does apply in this case with $c_j$ replaced by $c_j^{\prime}$.
\end{proof}

	 Figure \ref{shadow} shows a poset $X$ and the geometric realization  $|\mathcal{K}(X)|$. To illustrate this proof,
	  let $\sigma\in L_1(X)$ be represented by $(C, D)$.
	 Then the shaded area in the figure on the right is $W_0$.
\begin{figure}[htbp]
	\subfigure
	{\begin{minipage}{6cm}
			\centering
			\includegraphics[scale=0.65]{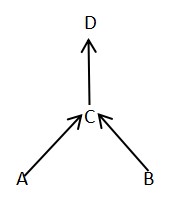}
	\end{minipage}}
	\subfigure
	{\begin{minipage}{6cm}
			\centering
			\includegraphics[scale=0.65]{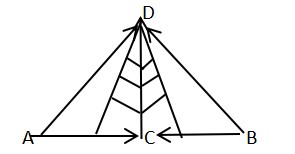}
	\end{minipage}}
	\caption{A poset $X$ and the geometric realization  $|\mathcal{K}(X)|$.}
	\label{shadow}
\end{figure}

Observe that an $(m, q)$-ball pair $(B^m, B^q)$ is unknotted by the argument in the proof of {\rm{\cite[Theorem 9]{ECZ1963}}}, provided that $B^m\searrow B^q$. If $\partial|\mathcal{K}(\sigma(Q_n))|\subseteq \partial|\mathcal{K}(X)|$ and ${\rm int}|\mathcal{K}(\sigma(Q_n))|\subseteq {\rm int}|\mathcal{K}(X)|$,
$$(|\mathcal{K}(X)|, |\mathcal{K}(\sigma(Q_n))|)$$ is a ball pair.
Thus by Lemma \ref{ballpair}, the following result holds.
\begin{Proposition}\label{abc}
Let $X$ be a contractible poset with $|\mathcal{K}(X)|$ being an $m$-manifold and $\sigma\in L_n(X)$. If ${\rm dim}\sigma(Q_n)<m$ and $(|\mathcal{K}(X)|, |\mathcal{K}(\sigma(Q_n))|)$ is a ball pair, then $$(|\mathcal{K}(X)|, |\mathcal{K}(\sigma(Q_n))|)$$ is an unknotted ball pair.
\end{Proposition}

The main result of this subsection is a direct consequence of the following lemma.
\begin{Lemma}\label{int}
Let $X$ be a contractible poset with $|\mathcal{K}(X)|$ being an $m$-manifold and   $\sigma_1\in L_n(X)$ with $|\mathcal{K}(\sigma_1(Q_n))|\subseteq {\rm int}|\mathcal{K}(X)|$.  Assume that ${\rm dim}\sigma_1(Q_n)<m$ and $(|\mathcal{K}(X)|, |\mathcal{K}(\sigma_1(Q_n))|)$ is a ball pair. Let $$c=dim|\mathcal{K}(X-\sigma_1(Q_n))|-1.$$ If $n\leq m-1$ and $p\leq c$,
then for any $\sigma_2\in L_p(X-\sigma_1(Q_n))$ with $|\mathcal{K}(\sigma_2(Q_p))|\subseteq {\rm int}|\mathcal{K}(X-\sigma_1(Q_n))|$ and $x\in \sigma_1(Q_n)$, the $p$th Betti number of
  $$X-\sigma_1(Q_n)-\sigma_2(Q_p)$$ for discrete cubical homology theory
  is nonzero if $p=m-2\leq c,~0<n<m-1$, or $p=m-1\leq c, ~n=0$, or $p=0$.
\end{Lemma}
\begin{proof}
By Corollary \ref{appro}, $$\psi_*: {H}_p^{\rm Cube}(X-\sigma_1(Q_n)-\sigma_2(Q_p))\rightarrow  {H}_p^{\rm Simpl}(\mathcal{K}(X-\sigma_1(Q_n)))$$ is surjective.  From Proposition \ref{abc}, we know that
$(|\mathcal{K}(X)|, |\mathcal{K}(\sigma_1(Q_n))|)$ in an unknotted ball pair. Then
\begin{equation*}
{H}_{q}^{\rm Sing}(|\mathcal{K}(X)|-|\mathcal{K}(\sigma_1(Q_n))|)=\begin{cases}
\mathbb{Z},~&\text{if}~q=m-2,~0<n<m-1, ~\text{or}\\
&~q=m-1, ~n=0, ~\text{or}~q=0,\\
0,~&\text{otherwise}.
\end{cases}
\end{equation*}
The conclusion of this lemma is then immediate from Lemma \ref{retract} and Theorem \ref{Main1}.
\end{proof}

We are now in a position to state and prove the main result of this subsection. Let $M=|K|$ be a triangulated $m$-manifold with $K$ simplicially collapsible. Then $\chi(K)$ is collapsible {\rm{\cite[Theorem 4.2.11]{JAB2011}}} and $|\mathcal{K}(\chi(K))|=|K|$ is an $m$-manifold.
Theorem \ref{Main1} and Lemma \ref{int} tell us when the contractibility is not preserved by the manifold $M$ with two specific balls removed.

\begin{Proposition}\label{apptom}
	Let $M=|K|$ be a triangulated $m$-manifold with $K$ simplicially collapsible and $\sigma_1\in L_n(\chi(K))$ with $|\mathcal{K}(\sigma_1(Q_n))|\subseteq {\rm int}M$. Assume also that ${\rm dim}\sigma_1(Q_n)<m$ and $(M, |\mathcal{K}(\sigma_1(Q_n))|)$ is a ball pair. Let $$c=dim|\mathcal{K}(\chi(K)-\sigma_1(Q_n))|-1.$$ If $n\leq m-1$,
	then for any $\sigma_2\in L_p(\chi(K)-\sigma_1(Q_n))$ with $|\mathcal{K}(\sigma_2(Q_p))|\subseteq {\rm int}|\mathcal{K}(\chi(K)-\sigma_1(Q_n))|$,
	$$M-|\mathcal{K}(\sigma_1(Q_n))|-|\mathcal{K}(\sigma_2(Q_n))|$$
	in not contractible, if $0<p=m-2\leq c,~0<n<m-1$, or $0<p=m-1\leq c, ~n=0$.
\end{Proposition}

The next example will help better understand the content of Proposition \ref{apptom}.
\begin{Example}
{\rm Let $X$ be the following poset, which is a finite model of the $3$-ball. 
\begin{equation*}
\xymatrix{&i&\\
	f\ar[ur]&g\ar[u] &h\ar[ul]\\
	c\ar[u]\ar[ur]& d\ar[ul]\ar[ur]&e\ar[u]\ar[ul] \\
	a\ar[u]& b\ar[u]\ar[ul]\ar[ur]&.
}
\end{equation*}
Let $\sigma_1\in L_0(X)$ be the constant map $Q_0\rightarrow i$, and let $\sigma_2\in L_1(X-\sigma_1(Q_0))$ be the represented $(b, c)$. Then $X-\sigma_1(Q_0)-\sigma_2(Q_1)$ is a finite model of the circle, being not contractible. It follows that the manifold $|\mathcal{K}(X)|-|\mathcal{K}(\sigma_1(Q_0))|-|\mathcal{K}(\sigma_2(Q_1))|$ is not contractible.
}
\end{Example}

\section{Discrete Hurewicz map}\label{section5}
The aim of this section is to construct an analogue of the Hurewicz homomorphism connecting discrete homotopy theory with discrete  cubical homology theory. 
The discrete Hurewicz homomorphism to be defined is similar to the one in {\rm{\cite{BL2019}}}. This discrete Hurewicz homomorphism indeed coincides with the classical one. This is stated in Theorem \ref{Hurewicz coincide}.

Given a finite poset $X$ with basepoint $x_0$,
let $f: (\mathbb{Z}^n, \partial\mathbb{Z}^n)\rightarrow (X, x_0)$ be a map of radius $r$. 
Define $f^{\prime}: (\mathbb{Z}^n, \partial \mathbb{Z}^n)\rightarrow (X, x_0)$ by $$f^{\prime}(x)=f(x_1-2r,\cdots, x_n-2r).$$
 Take $T_rf: (I_{4r}^n, \partial I_{4r}^n)\rightarrow (X, x_0)$ to be the restriction of $f^{\prime}$ to $I^n_{4r}$. 

To each map $f: I_{p_1}\times\cdots\times I_{p_n}\rightarrow X$, we associate an element $\varphi(f)$ of $\mathcal{C}_n^{\rm Cube}(X)$  in the following way. If $p_i=0$ for some $i$, then $\varphi(f)=0$. Assume $p_i>0$ for all $i$. For
each $x\in I_{p_1-1}\times\cdots\times I_{p_n-1}$, let $f^x: Q_n\rightarrow X$ be the $n$-cube given by $$f^x(y)=f(x+y),$$ 
and then $\varphi(f)$ is defined by
$$ \varphi(f)=\sum_{x\in I_{p_1-1}\times\cdots\times I_{p_n-1}}f^x.$$
\begin{Definition}{\rm
The \emph{discrete Hurewicz map} $h^D: \pi^{D}_n(X, x_0)\rightarrow {H}^{\rm Cube}_n(X)$ is defined by
$$h^D([f])=[\varphi(T_{r}f)],$$
where $f: (\mathbb{Z}^n, \partial\mathbb{Z}^n)\rightarrow (X, x_0)$ is a map of radius $r$.}
\end{Definition}
This definition is justified by the following lemma.

\begin{Lemma}\label{cccc}
\begin{itemize}
	\item [(1)] Assume $f: (I^n_r, \partial I_r^n)\rightarrow (X, x_0)$ is a map and $\widetilde{f}: (I^n_{\widetilde r}, \partial I_{\widetilde r}^n)\rightarrow (X, x_0)$ is an extension of $f$ defined by
	\begin{equation}\label{eq3}
	\widetilde{f}(x)=\begin{cases}
	f(x) ~&\text{if}~x\in I^n_r,\\
	x_0 ~&\text{if}~x\not\in I^n_r.
	\end{cases}
	\end{equation}
	Then $\varphi(f)=\varphi(\widetilde{f})$.
	\item[(2)] If $f: (\mathbb{Z}^n, \partial\mathbb{Z}^n)\rightarrow (X, x_0)$ is a map of radius $r$, then $$\varphi(T_rf)\in {\rm Ker}(\partial^{\rm Cube}_n).$$
	\item[(3)] If  $f$ and $g$ are homotopic maps $(\mathbb{Z}^n, \partial\mathbb{Z}^n)\rightarrow (X, x_0)$, then $$[\varphi(T_{r_f}f)]=[\varphi(T_{r_g}g)]$$ in ${H}^{\rm Cube}_n(X)$, where $r_f$ and $r_g$ are radii of $f$ and $g$, respectively.
\end{itemize}
\end{Lemma}
\begin{proof}
By the definition of $\varphi$,
\begin{equation*}
\begin{split}
\varphi(\widetilde{f})&=\sum_{x\in I^n_{r-1}}\widetilde{f}^x+\sum_{x\in I^n_{\widetilde r-1}-I^n_{r-1}}\widetilde{f}^x\\
&=\sum_{x\in I^n_{r-1}}f^x=\varphi(f),
\end{split}
\end{equation*}
where the second equality comes from the fact that $\widetilde{f}^x$ is a degenerate cube for each $I^n_{\widetilde r-1}-I^n_{r-1}$. This proves $(1)$.

For $1\leq i\leq n$, let $\Gamma_i=\{x\in I_{4r-1}^n~|~x_i=0\}$.   One has
\begin{equation*}
\begin{split}
\partial_n^{\rm Cube}\varphi(T_rf)&=\sum_{i=1}^{n}(-1)^i\sum_{x\in\Gamma_i}\sum_{r=0}^{4r-1}\left(f_i^-(T_rf)^{x+re_i}-f_i^+(T_rf)^{x+re_i}\right)\\
&=\sum_{i=1}^{n}(-1)^i\sum_{x\in\Gamma_i}\left(f_i^-(T_rf)^{x}-f_i^+(T_rf)^{x+(4r-1)e_i}\right), 
\end{split}
\end{equation*}
where $e_i\in\mathbb{Z}^n$ is a vector with $1$ in the $i$-th position and $0$  everywhere else.
Observe that for each $i$ and $x\in\Gamma_i$, $f_i^-(T_rf)^{x}-f_i^+(T_rf)^{x+(4r-1)e_i}$ is $0$ since $T_rf$ is constant on $\partial I_{4r}^n$. This proves $(2)$.
	
 Given a homotopy $H: \mathbb{Z}^n\times I_p\rightarrow X$  from $f$ to $g$, 
 let $TH$ be the map given by
 \begin{equation*}
 \begin{split}
I^n_{4r}\times I_p&\rightarrow X,\\
(x, t)&\mapsto H(x-2r, t),
 \end{split}
 \end{equation*}
 $r$ being taken the maximum integer of radii of $H_i, ~0\leq i\leq p$.
Let $$\Gamma_i=\{x\in I^n_{4r-1}\times I_{p-1}~|~x_i=0\}$$ and let $e_i\in \mathbb{Z}^{n+1}$ 
be the vector with $1$ in the $i$-th position and $0$  everywhere else. We have
\begin{equation*}
\begin{split}
\partial_{n+1}^{\rm Cube}\varphi(T H)(y)&=\sum_{i=1}^{n}(-1)^i\sum_{x\in\Gamma_i}\sum_{r=0}^{4r-1}\left(f_i^- (TH)^{x+re_i}(y)-f_i^+ (TH)^{x+re_i}(y)\right)\\
+&(-1)^{n+1}\sum_{x\in\Gamma_{n+1}}\sum_{r=0}^{p-1}\left(f_{n+1}^- (TH)^{x+re_{n+1}}(y)-f_{n+1}^+ (TH)^{x+re_{n+1}}(y)\right)\\
&= (-1)^{n+1}\sum_{x\in\Gamma_{n+1}}\left(f_{n+1}^- (TH)^{x}(y)-f_{n+1}^+ (TH)^{x+(p-1)e_{n+1}}(y)\right)\\
&=(-1)^{n+1}\sum_{x\in\Gamma_{n+1}} (TH)(x_1+y_1,\cdots, x_n+y_n,0)\\
-&(-1)^{n+1}\sum_{x\in\Gamma_{n+1}} (TH)(x_1+y_1,\cdots, x_n+y_n,p)\\
&=(-1)^{n+1}\left(\varphi({(TH)}_0)-\varphi({(TH)}_p)\right)(y)\\
&=(-1)^{n+1}\left(\varphi(T_{r_f}f)-\varphi(T_{r_g}g)\right)(y),
\end{split}
\end{equation*}
where the second equality follows just as in the proof of $(2)$ and
the last equality follows from $(1)$.
\end{proof}

From Lemma \ref{cccc}, we know that $h^D$ is well-defined. The next proposition will show that indeed $h^D$ preserves group structure.
\begin{Proposition}
The discrete Hurewicz map $h^D: \pi^{D}_n(X, x_0)\rightarrow {H}^{\rm Cube}_n(X)$ is a homomorphism.
\end{Proposition}
\begin{proof}
For any two maps $f, g: (\mathbb{Z}^n, \partial\mathbb{Z}^n)\rightarrow (X, x_0)$, $\varphi(T_{r_f+r_g}(f\cdot g))-(\varphi(T_{r_f}f)+\varphi(T_{r_g}g))$ is zero in $\mathcal{C}_n^{\rm Cube}(X)$, being a sum of degenerate $n$-cubes $Q_n\rightarrow x_0$, so $h^D([f]\cdot[g])=h^D([f])+h^D([g])$.
\end{proof}

There is a discrete Hurewicz theorem in dimension $1$. The same strategy in the proof {\rm{\cite[Theorem 4.1]{HB2014}}} applies in this case.
\begin{Theorem}\label{di1}
For any finite poset $X$, the map  
$$h^D: \pi^{D}_1(X, x_0)\rightarrow {H}^{\rm Cube}_1(X)$$
is surjective with ${\rm Ker}h^D=[\pi^{D}_1(X, x_0), \pi^{D}_1(X, x_0)]$.
\end{Theorem}

Let us recall the notion of classical Hurewicz map {\rm{\cite{AH2002}}}. Thinking of $\pi_n(X, x_0)$
as homotopy classes of maps $f: (L^n, \partial L^n)\rightarrow (X, x_0)$, the 
Hurewicz map $h: \pi_n(X, x_0)\rightarrow H_n^{\rm Sing}(X, x_0)$
is defined by $h([f])=f_*(\beta)$, where $\beta$ is a fixed generator of $H_n^{\rm Sing}(L^n, \partial L^n)=\mathbb{Z}$ and $f_*:H_n^{\rm Sing}(L^n, \partial L^n)\rightarrow H_n^{\rm Sing}(X, x_0)$ is induced by $f$.

\begin{Theorem}\label{Hurewicz coincide}
 For any $n\geq 0$, the following diagram commutes.
\end{Theorem}
\begin{equation*}
\xymatrix{  \pi_n^D(X, x_0)\ar[r]^{h^D} \ar[dd]_{\theta} &  H_n^{\rm Cube}(X)\ar[d]^{\psi}\\
            &H^{\rm Simpl}_n(\mathcal{K}(X))\ar[d]^{\eta_*}\\
\pi_n(X, x_0)\ar[r]^-{h}& H^{\rm Sing}_n(|\mathcal{K}(X)|).
}
\end{equation*}
\begin{proof}
Let $f:(\mathbb{Z}^n, \partial\mathbb{Z}^n)\rightarrow (X, x_0)$ be a map of radius $r$. By computing, we get
\begin{equation*}
\begin{split}
\eta_*\psi_* h^D([f])&=\eta_*\psi[\sum_{x\in I^{n}_{4r-1}}(T_rf)^x]\\
&=\eta_*\left[\sum_{x\in I^{n}_{4r-1}}\sum_{\tau\in S_n}{\rm sign}(\tau)\mathcal{K}\left(\{(T_rf)^x (p_{\tau}(i))~|~0\leq i\leq n\}\right)\right]\\
&=\eta_*\left[\sum_{x\in I^{n}_{4r-1}}\sum_{\tau\in S_n}{\rm sign}(\tau)\mathcal{K}\left(\{T_rf (x+p_{\tau}(i))~|~0\leq i\leq n\}\right)\right]\\
&=\eta_*\mathcal{K}(T_rf)_*\left[\sum_{x\in I^{n}_{4r-1}}\sum_{\tau\in S_n}{\rm sign}(\tau)\mathcal{K}\left(\{ x+p_{\tau}(i)~|~0\leq i\leq n\}\right)\right],
\end{split}
\end{equation*}
where $\left[\sum\limits_{x\in I^{n}_{4r-1}}\sum\limits_{\tau\in S_n}{\rm sign}(\tau)\mathcal{K}\left(\{ x+p_{\tau}(i)~|~0\leq i\leq n\}\right)\right]$ is a generator of $$H^{\rm Simpl}_n(\mathcal{K}(I^{n}_{4r}), \mathcal{K}(\partial I^{n}_{4r})),$$ denoted $\alpha$.

Consider an order-preserving map $P: \mathbb{Z}^n\rightarrow\mathbb{Z}^n$ given by
$$P(x)=x+2r.$$
Then we have $f(x)=T_rf(Px)$ for every $x\in\mathbb{Z}_{[-2r, 2r]}^n$. It follows that 
 $$|\mathcal{K}(f)|=|\mathcal{K}(T_rf)\circ \mathcal{K}(P)|$$ when restricted to $\mathbb{Z}_{[-2r, 2r]}^n$.
Denote by $\beta={\lambda_{2r}}_*|\mathcal{K}(P)|^{-1}_*\eta_*\alpha$ the fixed generator of
$H^{\rm Sing}_n(L^n, \partial L^n)$. Then
\begin{align*}
h\theta([f])&=(|\mathcal{K}(f)|\circ\lambda_{2r}^{-1})_*\eta_*(\beta)\\
&=|\mathcal{K}(T_rf)\circ \mathcal{K}(P)|_*\circ{(\lambda_{2r}^{-1})}_*({(\lambda_{2r})}_*|\mathcal{K}(P)|^{-1}_*\eta_*\alpha)\\
&=|\mathcal{K}(T_rf)|_*\eta_*\alpha\\
&=\eta_*\mathcal{K}(T_rf)_*\alpha.
\end{align*}

\end{proof}

\section*{Data availability}
This study has no associated data.
\section*{Declarations}
The authors declare that they have no conflict of interest.

\end{document}